\documentclass[11pt]{amsart}

\usepackage{amsmath}
\usepackage{amssymb}
\usepackage{amsthm}
\usepackage{graphicx}
\usepackage{float}
\usepackage{fancyhdr}
\usepackage[linktocpage=true]{hyperref}

\let\oldtocsection=\tocsection

\let\oldtocsubsection=\tocsubsection

\let\oldtocsubsubsection=\tocsubsubsection

\newtheorem{thm}{Theorem}
\newtheorem{rem}{Remark}
\newtheorem{defn}{Definition}
\newtheorem{lemma}[thm]{Lemma}
\newtheorem{prop}[thm]{Proposition}

\newcommand{\R}{\mathbb{R}}
\newcommand{\C}{\mathbb{C}}

\renewcommand{\tocsection}[2]{\hspace{0em}\oldtocsection{#1}{#2}}
\renewcommand{\tocsubsection}[2]{\hspace{2em}\oldtocsubsection{#1}{#2}}
\renewcommand{\tocsubsubsection}[2]{\hspace{5em}\oldtocsubsubsection{#1}{#2}}

\numberwithin{equation}{section}

\makeindex

\begin{document}

\title{Transformations of harmonic bundles and Willmore surfaces}

\author{A. C. Quintino}

\address{A. C. Quintino, CENTRO DE MATEM\'{A}TICA, APLICA\c{C}\~{O}ES FUNDAMENTAIS E INVESTIGA\c{C}\~{A}O OPERACIONAL, FACULDADE DE CI\^{E}NCIAS DA UNIVERSIDADE DE LISBOA\\1749-016 LISBOA\\PORTUGAL}
\email{amquintino@ciencias.ulisboa.pt} 

\begin{abstract}
Willmore surfaces are the extremals of the Willmore functional
(possibly under a constraint on the conformal structure). With the
characterization of Willmore surfaces by the (possibly \emph{perturbed}) harmonicity of the mean
curvature sphere congruence \cite{blaschke,burstall+calderbank,ejiri,rigoli}, a
zero-curvature formulation follows \cite{burstall+calderbank}.
Deformations on the level of harmonic maps prove to give rise to
deformations on the level of surfaces, with the definition of a
spectral deformation \cite{burstall+calderbank,SD} and of a
B\"{a}cklund transformation \cite{BQ} of Willmore surfaces into
new ones, with a Bianchi permutability between the two \cite{BQ}. This
text is dedicated to a self-contained account of the topic, from a conformally-invariant viewpoint, in Darboux's light-cone model of the conformal $n$-sphere. 
\end{abstract}

\maketitle

\setcounter{tocdepth}{4}

\tableofcontents

\section{Introduction}

Among the classes of Riemannian submanifolds, there is that of
\textit{Willmore surfaces}, named after T. Willmore \cite{willmore} (1965),
although the topic was mentioned by Blaschke \cite{blaschke}
(1929) and by Thomsen \cite{thomsen} (1923), as a variational
problem of optimal geometric realization of a given compact surface in $3$-space regarding   the minimization of some natural energy.

Early in the nineteenth century, Germain \cite{germain1, germain2} studied elastic surfaces. On her pioneering
analysis, she claimed that the elastic force of a thin plate is
proportional to its mean curvature, $H=(k_{1}+k_{2})/2$, for $k_{1}$ and $k_{2}$ the maximum and minimum curvatures among all intersections of the
surface with perpendicular planes, at each point. Since then, the mean curvature
remains a key concept in the theory of elasticity. 

In modern literature on the elasticity of membranes, a weighted sum of the total mean
curvature, the total squared mean curvature and the total Gaussian
curvature is considered the elastic energy of a membrane. By
neglecting the total mean curvature, by physical considerations, and
having in consideration the Gauss-Bonnet Theorem, T. Willmore defined
the \textit{Willmore energy} of a compact oriented surface
$\Sigma$, without boundary, isometrically immersed in $\mathbb R^{3}$ to be $$\mathcal{W}=\int
_{\Sigma}H^{2}dA,$$ averaging the mean curvature square over the
surface. 

From the perspective of energy extremals, the Willmore functional may be extended to isometric immersions $\phi$ of compact oriented surfaces $\Sigma$ in a
general Riemannian manifold $M$ of constant sectional
curvature by means of half (or any other scale) of the total squared norm of
$$\Pi _{0}=\Pi
-g_{\phi}\otimes\mathcal{H},$$
the trace-free part of the second fundamental form $\Pi$, for $\mathcal{H}=\frac{1}{2}\mathrm{tr}(\Pi)$, the mean curvature vector, and $g_{\phi}$ the metric induced in $\Sigma$ by $\phi$. 
In fact, given $(X_{i})_{i=1,2}$ a local orthonormal frame of
$T\Sigma$, the Gauss equation, relating the curvature tensors of $\Sigma$ and $M$, establishes, in particular, 
$$K-\hat{K}=(\Pi(X_{1},X_{1}),\Pi(X_{2},X_{2})-(\Pi(X_{1},X_{2}),\Pi(X_{1},X_{2}),$$
for $K$ the Gaussian curvature of $\Sigma$
and $\hat{K}$ the sectional curvature of $M$, 
and, therefore,
$$|\Pi _{0}|^2=\sum _{i,j}(
\Pi _{0}(X_{i},X_{j}),\Pi
_{0}(X_{i},X_{j}))=2(|\mathcal{H}|^2-K+\hat{K}).$$Hence, for the particular case of surfaces in $\R^{3}$, the two functionals 
share critical points. 

Willmore surfaces are the extremals of the Willmore
functional. \textit{Constrained Willmore surfaces} appear as the
generalization of Willmore surfaces that arises when we consider
extremals of the Willmore functional with respect to \textit{infinitesimally
conformal variations}, rather than with respect to all variations. The Euler-Lagrange equations include a Lagrange multiplier. Willmore surfaces are the constrained Willmore surfaces admitting the zero multiplier.  The
zero multiplier is not necessarily the only multiplier for a
constrained Willmore surface with no constraint on the conformal
structure, though. In fact, the uniqueness of multiplier
characterizes \cite{BQ} non-isothermic constrained Willmore surfaces. Constant mean curvature surfaces in $3$-dimensional space-forms are examples of isothermic constrained Willmore surfaces, as proven by J. Richter \cite{richter}. A classical result by Thomsen \cite{thomsen} characterizes
isothermic Willmore surfaces in $3$-space as minimal surfaces in
some $3$-dimensional space-form. 

It is well-known that the Levi-Civita connection is not a conformal invariant. In fact (see, for example, \cite[Section~3.12]{willmore2}, under a conformal change $g'=e^{2u}g$ of a metric $g$ on $\Sigma$, for some $u\in C^{\infty}(\Sigma,\R)$, the Levi-Civita connections $\nabla$ and $\nabla '$ on $(\Sigma,g)$ and $(\Sigma,g')$, respectively, are related by
$$\nabla'_{X}Y=\nabla_{X}Y+(Xu)Y+(Yu)X-g(X,Y)(du)^{*},$$
for $(du)^{*}$ the
contravariant form of $du$ with respect to $g$, for all $X,Y\in\Gamma(T\Sigma)$. It follows  that, under a conformal change of metric on a Riemannian manifold $M$, the second fundamental form of an isometric immersion $\phi:\Sigma\rightarrow M$ changes according to 
\begin{equation}\label{eq:Pichanges}
\Pi'(X,Y)=\Pi(X,Y)-g_{\phi}(X,Y)\pi_{N_{\phi}}(\phi^{*}(du)^{*}),
\end{equation}
for $\pi_{N_{\phi}}$ the orthogonal projection of the pull-back bundle $\phi^{*}TM$ onto the normal bundle $N_{\phi}=(d\phi(T\Sigma))^{\perp}$ and $\phi^{*}(du)^{*}$ the
pull-back by $\phi$ of $(du)^{*}$; and, therefore, 
$$\mathcal{H}'=e^{-2u\circ\phi}\mathcal{H}-e^{-2u\circ\phi}\pi_{N_{\phi}}(\phi^{*}(du)^{*}),$$
relating the respective mean curvature vectors. 
Hence, under a conformal change of the metric, the trace-free part of the second fundamental form remains invariant, so that its
squared norm and the area element change in an inverse way, leaving
the Willmore energy unchanged. In particular, this establishes the 
class of constrained Willmore surfaces as a conformally-invariant
class.

Conformal invariance motivates us to move from Riemannian to conformal geometry. Our study is one of surfaces in $n$-dimensional space-forms from a conformally-invariant viewpoint. For this, we find a convenient setting in Darboux's light-cone model \cite{darbouxsphere} of the conformal $n$-sphere, viewing the $n$-sphere
not as the round sphere in the Euclidean space $\R^{n+1}$ but as the
celestial sphere in the Lorentzian spacetime $\R^{n+1,1}$. 

A manifestly conformally-invariant formulation of the Willmore energy is presented, following the definition presented by Burstall, Ferus, Leschke, Pedit and Pinkall \cite{BFLPP}, in the
quaternionic setting, for the particular case of $n=4$.

A fundamental construction in conformal geometry of surfaces is the
\textit{mean curvature sphere congruence}, the bundle of $2$-spheres tangent
to the surface and sharing mean curvature vector with it at each
point (although the mean curvature vector is not
conformally-invariant, under a conformal change of the metric it
changes in the same way for the surface and the osculating
$2$-sphere). From the early twentieth century, with the work of Blaschke \cite{blaschke}, the family of mean curvature spheres has been known as the \textit{central sphere
congruence}. Nowadays, after Bryant's paper \cite{bryant}, it goes as well by the name \textit{conformal Gauss map}.

A key result by Blaschke \cite{blaschke} ($n=3$) and, independently, Ejiri \cite{ejiri} and Rigoli \cite{rigoli} (general $n$) characterizes Willmore surfaces by the
harmonicity of the central sphere congruence. The well-developed theory of harmonic maps, and, in particular, the integrable systems approach to these, then applies. The starting point is the fact that, for a map into a Grassmannian, harmonicity amounts to the flatness of a certain family of connections depending on a spectral parameter, according to Uhlenbeck \cite{uhlenbeck}. A zero-curvature characterization of Willmore surfaces follows. This characterization generalizes to constrained Willmore surfaces, as established by Burstall and Calderbank \cite{burstall+calderbank}.

The zero-curvature representation of the harmonic map equations allows
one to deduce two kinds of symmetry: harmonic maps admit a \emph{spectral
deformation} \cite{uhlenbeck}, by exploiting a scaling freedom in the spectral parameter, 
and \emph{B\"{a}cklund transformations}, which
arise by applying chosen gauge transformations to the family of
flat connections, as studied by Terng and Uhlenbeck \cite{terng+uhlenbeck,uhlenbeck}. Aiming to apply this theory to constrained Willmore surfaces, and in order to address the possibly non-harmonic central sphere congruences of constrained
Willmore surfaces, the notion of
\textit{perturbed harmonicity} for a map into a Grassmannian is introduced \cite{BQ}. It applies to the central sphere congruence and it provides a characterization of constrained Willmore surfaces. 

A spectral deformation and B\"{a}cklund transformations of perturbed harmonic maps into new ones are defined \cite{BQ}. Some care is required to see that, when applied to the central sphere congruence of a constrained Willmore surface, each new map still is the central sphere congruence of a surface. Deformations on the level of perturbed harmonic maps prove \cite{BQ} to give rise to deformations on the level of surfaces, with the definition of a spectral deformation and of B\"{a}cklund transformations of constrained Willmore 
surfaces into new ones. This spectral deformation of constrained Willmore surfaces coincides, up to reparametrization, with the one presented by Burstall, Pedit and Pinkall \cite{SD}, in terms of the \emph{Schwarzian derivative} and the \emph{Hopf differential}, later defined by the action of a loop of flat connections, by Burstall and Calderbank \cite{burstall+calderbank}. 

The class of constrained Willmore surfaces is in this way established
as a class of surfaces with strong links to the theory of integrable
systems, admitting a spectral deformation and a B\"{a}cklund transformation, with a Bianchi permutability between the two, as proven in \cite{BQ}. All these transformations corresponding to the zero multiplier preserve the class of Willmore
surfaces. 

The isothermic surface condition is known \cite{SD} to be preserved
under constrained Willmore spectral deformation. As
for B\"{a}cklund transformation of isothermic constrained Willmore surfaces, we
believe it does not necessarily preserve the isothermic condition.
In contrast, the constancy of the mean curvature of a surface in
$3$-dimensional space-form is preserved by both constrained Willmore
spectral deformation, cf. \cite{SD}, and constrained Willmore B\"{a}cklund transformation, cf. \cite{thesis}, for special
choices of parameters, with preservation of both the space-form and
the mean curvature in the latter case. However, constant mean curvature surfaces are not conformally-invariant objects, requiring that we carry a distinguished space-form. This shall be the subject of a forthcoming paper.

\section{Constrained Willmore surfaces and perturbed harmonicity}

Consider $\underline{\C}^{n+2}=\Sigma\times(\R^{n+1,1})^{\C}$
provided with the complex bilinear extension of the metric on
$\underline{\R}^{n+1,1}$. In what follows, we shall make no explicit
distinction between a bundle and its complexification, and move from
real tensors to complex tensors by complex multilinear extension,
with no need for further reference, preserving notation.

Throughout this text, we will consider the identification
$$\wedge^{2}\R^{n+1,1}\cong o(\R^{n+1,1})$$ of the exterior power
$\wedge^{2}\R^{n+1,1}$ with the orthogonal algebra $o(\R^{n+1,1})$
via $$u\wedge v(w):=(u,w)v-(v,w)u$$ for $u,v,w\in\mathbb R^{n+1,1}$. Given $\mu,\eta\in\Omega^{1}(\Sigma\times o(\R^{n+1,1}))$, we use $[\mu\wedge\eta]$ to denote the $2$-form defined from the Lie Bracket $[\,\,,\,]$ in $o(\R^{n+1,1})$:
$$[\mu\wedge\eta]_{(X,Y)}=[\mu_{X},\eta_{Y}]-[\mu_{Y},\eta_{X}],$$for all $X,Y\in\Gamma(T\Sigma)$. We consider the bundle $\mathrm{End}(\underline{\R}^{n+1,1})$, and,
more generally, any bundle of morphisms, provided with the metric
defined by $(\xi, \eta):=\mathrm{tr}(\eta^{t}\xi)$ and we shall move
from a connection on $\underline{\R}^{n+1,1}$ to a connection on
$\mathrm{End}(\underline{\R}^{n+1,1})$ via $\nabla
\xi=\nabla\circ\xi-\xi\circ\nabla$, with preservation of notation.
Note that, in the case of a metric connection $\nabla$ on
$\underline{\R}^{n+1,1}$, we have $$\nabla (u\wedge v)=\nabla u\wedge
v+u\wedge \nabla v,$$ for all $u,v\in\Gamma(\underline{\R}^{n+1,1})$.

Our theory is local and, throughout this text, with no need for
further reference, restriction to a suitable non-empty open set
shall be underlying.

\subsection{Real constrained Willmore surfaces}
\subsubsection{Conformal geometry of the sphere}
Our study is one of surfaces in the conformal $n$-sphere, with
$n\geq 3$, in Darboux's light-cone model \cite{darbouxsphere} of the latter. For this, contemplate the light-cone $\mathcal{L}$ in the Lorentzian vector
space $\R^{n+1,1}$ and its projectivization
$\mathbb{P}(\mathcal{L})$, provided with the conformal structure
defined by a metric $g_{\sigma}$ arising from a never-zero section
$\sigma$ of the tautological bundle
$\pi:\mathcal{L}\rightarrow\mathbb{P}(\mathcal{L})$ via
$$g_{\sigma}(X,Y)=(d\sigma (X), d\sigma(Y)).$$
For $v_{\infty}\in\R^{n+1,1}_{\times}$, set
$$S_{v_{\infty}}:=\{v\in\mathcal{L}:(v,v_{\infty})=-1\},$$ an
$n$-dimensional submanifold $\mathbb R^{n+1,1}$. Given $v\in S_{v_{\infty}}$, 
\begin{equation}\label{eq:TSvinfty}
T_{v}S_{v_{\infty}}=<v,v_{\infty}>^{\perp}.
\end{equation}
The fact that $(v,v_{\infty})\neq 0$ establishes the non-degeneracy of the subspace $<v,v_{\infty}>$ of $\R^{n+1,1}$, establishing a decomposition 
\begin{equation}\label{eq:TSinfty}
\R^{n+1,1}=<v,v_{\infty}>\oplus T_{v}S_{v_{\infty}}.
\end{equation}
In its turn, the nullity of $v$ establishes $<v,v_{\infty}>$ as a $2$-dimensional space with a metric with signature $(1,1)$, showing that $S_{v_{\infty}}$ inherits from
$\R^{n+1,1}$ a positive definite metric. Furthermore: for $v_{\infty}$
non-null, orthoprojection onto $\langle v_{\infty}\rangle^{\perp}$
induces an isometry between $S_{v_{\infty}}$ and $\{v\in\langle
v_{\infty}\rangle^{\perp}:(v,v)=-1/(v_{\infty},v_{\infty})\}$,
whereas, when $v_{\infty}$ is null, for any choice of $v_{0}\in
S_{v_{\infty}}$, orthoprojection onto $\langle
v_{0},v_{\infty}\rangle^{\perp}$ restricts to an isometry of
$S_{v_{\infty}}$. We conclude that $S_{v_{\infty}}$ inherits from
$\R^{n+1,1}$ a positive definite metric of (constant) sectional
curvature $-(v_{\infty},v_{\infty})$, defining a copy of the $n$-sphere, a copy of Euclidean $n$-space or two copies of hyperbolic $n$-space, according to the sign of $(v_{\infty},v_{\infty})$. 

By construction, the bundle projection $\pi$
restricts to give a conformal diffeomorphism
$$\pi_{\vert{S_{v_{\infty}}}}:S_{v_{\infty}}\rightarrow
\mathbb{P}(\mathcal{L})\backslash\mathbb{P}(\mathcal{L}\cap\langle
v_{\infty}\rangle^{\perp}).$$ In particular, choosing $v_{\infty}$
time-like identifies $\mathbb{P}(\mathcal{L})$ with the conformal
$n$-sphere, $$S^{n}\cong\mathbb{P}(\mathcal{L}).$$
This model linearizes 
the conformal geometry of the sphere. For example, $k$-spheres in
$S^{n}$ are identified with
$(k+1,1)$-planes $V$ in $\mathbb R^{n+1,1}$ via
$V\mapsto\mathbb{P}(\mathcal{L}\cap
V)\subset\mathbb{P}(\mathcal{L})$.

\subsubsection{Surfaces in the light-cone picture and central sphere congruence}

For us, a map $\Lambda:\Sigma\rightarrow\mathbb{P}(\mathcal{L})$ is the same as a null line subbundle of the trivial bundle
$\underline{\R}^{n+1,1}=\Sigma\times\R^{n+1,1}$, in the natural way. From this point of view, sections of $\Lambda$ are simply lifts of $\Lambda$ to maps $\Sigma\rightarrow \R^{n+1,1}$. Given such a $\Lambda$, we define
$$\Lambda^{(1)}:=\langle\sigma,d\sigma(T\Sigma)\rangle,$$for
$\sigma$ a lift of $\Lambda$. For further reference, note that $\Lambda$ is an immersion if
and only if the bundle $\Lambda^{(1)}$ has rank $3$.

Let then $\Lambda:\Sigma\rightarrow\mathbb{P}(\mathcal{L})$ be an immersion of an oriented
surface $\Sigma$, which we provide with the conformal structure
$\mathcal{C}_{\Lambda}$ induced by $\Lambda$ and with $J$ the
canonical complex structure (that is, $90^{\circ}$ rotation in the
positive direction in the tangent spaces, a notion that is obviously
invariant under conformal changes of the metric). Observe that every lift of $\sigma:\Sigma\rightarrow \R^{n+1,1}$ of $\Lambda$ is conformal: given $z$ a holomorphic chart of $\Sigma$, $(\sigma_{z},\sigma_{z})=0$ (or, equivalently, $(\sigma_{\bar{z}},\sigma_{\bar{z}})=0$). Set
\begin{equation}\label{eq:Lambdaij}
\Lambda^{1,0}:=\Lambda\oplus
d\sigma(T^{1,0}\Sigma),\,\,\,\,\Lambda^{0,1}:=\Lambda\oplus
d\sigma(T^{0,1}\Sigma),
\end{equation}
independently of the choice of a lift $\sigma$ of $\Lambda$,
defining in this way two complex rank $2$ subbundles of
$\Lambda^{(1)}$, complex conjugate of each other, $$\Lambda^{0,1}=\overline{\Lambda^{1,0}}.$$ The nullity and
conformality of the lifts of $\Lambda$ establish the isotropy of (both) $\Lambda^{1,0}$ (and $\Lambda^{0,1}$), whilst the fact that
$\Lambda$ is an immersion establishes that $\Lambda^{1,0}$ and
$\Lambda^{0,1}$ intersect in $\Lambda$, 
$$\Lambda^{1,0}\cap\Lambda^{0,1}=\Lambda.$$
Let $S:\Sigma\rightarrow \mathcal{G}:=\mathrm{Gr}_{(3,1)}(\mathbb
R^{n+1,1})$ be the \textit{central sphere congruence} of $\Lambda$,
$$S=\Lambda^{(1)}\oplus\langle\triangle\sigma\rangle=\langle\sigma,\sigma_{z},\sigma_{\bar{z}},\sigma_{z\bar{z}}\rangle,$$
for $\sigma$ any lift of $\Lambda$, $\triangle\sigma$ the Laplacian of
$\sigma$ with respect to the metric $g_{\sigma}$ and $z$ a
holomorphic chart of $\Sigma$. We use $\pi_{S}$ and $\pi_{S^{\perp}}$
to denote the orthogonal projections of $\underline{\mathbb R}^{n+1,1}$
onto $S$ and $S^{\perp}$, respectively.

Given $z$ a holomorphic chart of $\Sigma$, let $g_{z}$ denote the metric induced in $\Sigma$
by $z$. Differentiation of $(\sigma, \sigma_{z})=0$ gives $(\sigma,\sigma_{z\bar{z}})=-(\sigma_{z},\sigma_{\bar{z}})$, which the conformality of $z$ proves to be never-zero. In many occasions, it will be useful to consider a
special choice of lift of $\Lambda$, the \textit{normalized} lift
with respect to $z$, the section $\sigma^{z}:\Sigma\rightarrow
\mathcal{L}^{+}$ of $\Lambda$ (given a choice $\mathcal{L}^{+}$ of
one of the two connected components of $\mathcal{L}$) defined by
$g_{\sigma^{z}}=g_{z}$. For further reference, note that this
condition establishes, in particular, that
$(\sigma_{z}^{z},\sigma_{\bar{z}}^{z})$ is constant,
$(\sigma_{z}^{z},\sigma_{\bar{z}}^{z})=\frac{1}{2}$, and, therefore,
\begin{equation}\label{Hillseq}
\pi_{S}\sigma_{zz}^{z}\in\Gamma((\Lambda^{(1)})^{\perp}\cap
S)=\Gamma\Lambda.
\end{equation}

Consider the decomposition of the trivial flat connection $d$ on
$\underline{\mathbb R}^{n+1,1}$ as
$$d=\mathcal{D}\oplus\mathcal{N}$$ for
$\mathcal{D}$ the connection given by the sum of the connections
$\nabla^{S}$ and $\nabla^{S^{\perp}}$ induced on $S$ and
$S^{\perp}$, respectively, by $d$. Note that $\mathcal{D}$ is a
metric connection and, therefore, $\mathcal{N}$ is skew-symmetric,
$$\mathcal{N}\in\Omega^{1}(S\wedge S^{\perp}).$$ Note that, given
$\xi\in\Gamma(S\wedge S^{\perp})$, the transpose of $\xi\vert_{S}$
is $-\xi\vert_{S^{\perp}}$, and define a bundle isomorphism $S\wedge
S^{\perp}\rightarrow \mathrm{Hom}(S,S^{\perp})$ by
$\eta\mapsto\eta\vert_{S}$. Together with the canonical
identification of $S^{*}T\mathcal{G}$ and
$\mathrm{Hom}(S,S^{\perp})$, via $X\mapsto
(\rho\mapsto\pi_{S^{\perp}}(d_{X}\rho))$, this gives an
identification
\begin{equation}\label{eq:fsg}
S^{*}T\mathcal{G}\cong\mathrm{Hom}(S,S^{\perp})\cong S\wedge
S^{\perp},
\end{equation}
of bundles provided with the canonical metrics and connections (for
the connection $\mathcal{D}$ on $\underline{\R}^{n+1,1}$), (see, for
example, \cite{burstall+rawnsley}), which we will consider
throughout. Observe that, under the identification \eqref{eq:fsg}, we have 
\begin{equation}\label{eq:dSmathcalN}
dS=\mathcal{N}.
\end{equation}

We restrict our study to surfaces in $S^{n}$ which are not contained in any subsphere of $S^{n}$. This ensures, in particular, that, given $v_{\infty}\in\R^{n+1,1}$ non-zero, $\Lambda(\Sigma)\nsubseteq \mathbb{P}(\mathcal{L}\cap\langle v_{\infty}\rangle^{\perp})$: if $v_{\infty}$ is space-like, $\mathbb{P}(\mathcal{L}\cap\langle v_{\infty}\rangle^{\perp})$ is an hypersphere in $S^{n}$, whilst, in the case $v_{\infty}$ is time-like or light-like, this is always necessary the case. Hence, given $\sigma$ a lift of $\Lambda$, we have $(\sigma,v_{\infty})\neq 0$ and we define a local immersion 
$$\sigma_{\infty}:=(\pi_{\vert{S_{v_{\infty}}}})^{-1}\circ\Lambda=\frac{-1}{(\sigma,v_{\infty})}\,\sigma:\Sigma\rightarrow S_{v_{\infty}},$$ of $\Sigma$ into the 
space-form $S_{v_{\infty}}$, conformally diffeomorphic to the surface $\Lambda$. The normal  bundle $N_{\infty}$ to $\sigma_{\infty}$ can be identified with the normal bundle $S^{\perp}$ to the central sphere congruence of $\Lambda$, as bundles provided with metrics and connections:

\begin{lemma}\cite{SD}\label{NperpisoSperp}
Let $\mathcal{H}_{\infty}$ denote the mean curvature vector of $\sigma_{\infty}$. Then 
$$\xi\mapsto
\xi+(\xi,\mathcal {H_{\infty}})\sigma _{\infty}$$ defines an isomorphism 
$$\mathcal
{Q}:N_{\infty}\rightarrow S^{\perp},$$ of bundles provided
with a metric and a connection. Furthermore:
\begin{equation}\label{eq:mcvinSperp}
\mathcal
{Q}(\mathcal {H}_{\infty})=-\pi _{S^{\perp}}(v_{\infty}).
\end{equation}
\end{lemma}

\begin{proof}
Let $g_{\infty}$ denote the metric induced in $\Sigma$ by $\sigma_{\infty}$ and $\nabla^{N_{\infty}}$ denote the connection induced in $N_{\infty}$ by the pull-back connection by $\sigma_{\infty}$ of the Levi-Civita connection on $(T\Sigma, g_{\infty})$. According to \eqref{eq:TSvinfty} and \eqref{eq:TSinfty}, the pull-back bundle by $\sigma_{\infty}$ of $T\Sigma$ consists of the orthogonal complement in $\underline {\mathbb{R}}^{n+1,1}$ of the non-degenerate bundle $\langle\sigma _{\infty},v_{\infty}\rangle$, 
$$\sigma_{\infty}^{*}TS_{v_{\infty}}=\langle\sigma _{\infty},v_{\infty}\rangle
^{\perp}.$$
Let $\pi _{N_{\infty}}$ denote the orthogonal
projection of 
$$\underline {\mathbb{R}}^{n+1,1}=d\sigma _{\infty}(T\Sigma)\oplus N
_{\infty}\oplus \langle v_{\infty},\sigma _{\infty}\rangle$$
onto
$N_{\infty}$.  Since the metric in $S_{v_{\infty}}$ is the one
inherited from $\mathbb{R}^{n+1,1}$, the second fundamental form $\Pi _{\infty}$ of $\sigma_{\infty}$ is simply given
by
$$\Pi _{\infty}(X,Y)=\pi _{N_{\infty}}(d_{X}d_{Y}\sigma
_{\infty}),$$ for $X,Y\in \Gamma (T\Sigma)$, so that, given $\xi\in\Gamma
(N_{\infty})$ and $(e_{i})_{i}$ an orthonormal frame of
$(T\Sigma,g_{\infty})$, we have $(\xi,\sum _{i} d_{e_{i}}d_{e_{i}}\sigma
_{\infty})=2(\xi,\mathcal {H}_{\infty})$ and, therefore,
$$(\xi+(\xi,\mathcal {H_{\infty}})\sigma _{\infty},\sum _{i}
d_{e_{i}}d_{e_{i}}\sigma _{\infty})=0.$$ Together with the fact that
$N_{\infty}\subset\langle\sigma _{\infty},v_{\infty}\rangle
^{\perp}$,  this shows that $\xi+(\xi,\mathcal {H_{\infty}})\sigma
_{\infty}$ is, in fact, a section of $S^{\perp}$.

Clearly, $\mathcal {Q}$ is isometric, and, therefore, injective, as
$N_{\infty}$ is non-degenerate. Now
$\mathrm {rank}\, N_{\infty}=n-2=\mathrm {rank}\,S^{\perp}$ shows that $\mathcal
{Q}$ is an isometric isomorphism. Furthermore, given $\xi\in\Gamma
(N_{\infty})$,
$$ \nabla
^{S^{\perp}}(\mathcal {Q}(\xi))=\pi
_{S^{\perp}}(d\xi)+d(\xi,\mathcal {H_{\infty}})\pi
_{S^{\perp}}(\sigma_{\infty})+(\xi,\mathcal {H_{\infty}})\pi
_{S^{\perp}}(d\sigma _{\infty})=\pi _{S^{\perp}}(d\xi),$$ whilst
\begin{eqnarray*}
\mathcal {Q}(\nabla ^{N_{\infty}}\xi)=\pi _{N_{\infty}}(d\xi)+(\pi
_{N_{\infty}}(d\xi),\mathcal {H_{\infty}})\sigma _{\infty}\in\Gamma
(S^{\perp}).
\end{eqnarray*}
To conclude that $\mathcal {Q}$ preserves connections, we just need to
verify that $$d\xi -\pi _{N_{\infty}}(d\xi)\in\Gamma (S).$$ That is
immediate: $d\xi$ is
still a section of $\langle\sigma _{\infty},v_{\infty}\rangle
^{\perp}$,
$$(d\xi,\sigma_{\infty})=(d\xi,\sigma_{\infty})+(\xi,d\sigma_{\infty})=0=(d\xi,v_{\infty})+(\xi,dv_{\infty})=(d\xi,v_{\infty});$$
and, therefore, $d\xi -\pi _{N_{\infty}}(d\xi)$ is the orthogonal projection of $d\xi$ onto the tangent bundle to $\sigma_{\infty}$. 

Finally, the fact that 
$$(\mathcal {Q}(\xi ),\pi _{S^{\perp}}(v_{\infty}))=(\xi ,v_{\infty})+(\xi ,\mathcal
{H}_{\infty})(\sigma _{\infty},v_{\infty})= -(\mathcal
{Q}(\xi),\mathcal {Q}(\mathcal {H}_{\infty})),$$for $\xi\in N_{\infty}$, 
establishes \eqref{eq:mcvinSperp} and completes the proof.
\end{proof}

\subsubsection{The Willmore energy}
Suppose, for the moment, that $\Sigma$ is compact. The \textit{Willmore energy} $\mathcal{W}(\Lambda)$ of $\Lambda$ is given by\footnote{In the literature, different scalings of the Willmore energy can be found. Our choice is justified by the classical scaling in the Dirichlet energy functional.} 
$$\mathcal{W}(\Lambda)=\int _{\Sigma}|\Pi _{0}|^2dA,$$ for $\Pi_{0}$ the trace-free part of the second fundamental form of $\Lambda$ (calculated with respect to any representative metric on $S^{n}$ and independent of that choice). 

Next we present a manifestly conformally-invariant formulation of the Willmore energy.
It follows the definition presented in \cite{BFLPP}, in the
quaternionic setting, for the particular case of $n=4$. The intervention of the conformal structure will restrict to the Hodge $*$-operator, which is conformally-invariant
on $1$-forms over a surface. 

Given $\mu,\eta\in\Omega^{1}(S^{*}T\mathcal{G})$, let $(\mu\wedge\eta)$ be the $2$-form defined from the metric on $S^{*}T\mathcal{G}$:
$$(\mu\wedge\eta)_{(X,Y)}=(\mu_{X},\eta_{Y})-(\mu_{Y},\eta_{X}),$$for all $X,Y\in\Gamma(T\Sigma)$. Note that $$(dS\wedge *dS)=-(*dS\wedge dS)=(dS,dS)dA,$$ $(dS\wedge *dS)$ is a conformally invariant way of writing $(dS,dS)_{g}dA_{g}$, for $g\in\mathcal{C}_{\Lambda}$, with $dA_{g}$ denoting the area element of $(\Sigma,g)$ and $(,)_{g}$ denoting the Hilbert-Schmidt metric on $L((T\Sigma,g),S^{*}T\mathcal{G})$.

\begin{thm}\cite{BFLPP}\label{Wenergy}
$$\mathcal{W}(\Lambda)=\frac{1}{2}\int_{\Sigma}(d S\wedge *d S).$$
\end{thm}

\begin{proof}
By \eqref{eq:dSmathcalN}, fixing a metric on $\Sigma$,
$|dS|^{2}=|\mathcal{N}|^{2}$. To prove the theorem, we fix
$v_{\infty}\in\R^{n+1,1}$ non-zero, provide $\Sigma$ with the metric
induced by $\sigma_{\infty}$ and show that $|\mathcal{N}|^{2}=2|\Pi
_{\infty}|^2$, for $\Pi _{\infty}$ the trace-free part of
the second fundamental form of $\sigma_{\infty}$.

Fixing a local orthonormal frame $\{X_{i}\}_{i}$ of $T\Sigma$, we
have
$$|\mathcal{N}|^{2}=-\sum_{i}\mathrm{tr}(\mathcal{N}_{X_{i}}\mathcal{N}_{X_{i}})=2\sum_{i}\mathrm{tr}(\mathcal{N}_{X_{i}}^{t}\mathcal{N}_{X_{i}}\vert_{S}).$$
Recall that if $(e_{i})_{i}$ and $(\hat{e}_{i})_{i}$ are dual basis
of a vector space $E$ provided with a metric $(,)$, then, given
$\mu\in\mathrm{End}(E)$,
$\mathrm{tr}(\mu)=\sum_{i}(\mu(e_{i}),\hat{e}_{i})$. Let
$\hat{\sigma}_{\infty}$ be the section of $S$ determined by the
conditions $(\hat{\sigma}_{\infty},\hat{\sigma}_{\infty})=0$,
$(\sigma_{\infty},\hat{\sigma}_{\infty})=-1$ and
$(\hat{\sigma}_{\infty}, d\sigma_{\infty})=0$. Then
$(\sigma_{\infty},
d_{X_{1}}\sigma_{\infty},d_{X_{2}}\sigma_{\infty},\hat{\sigma}_{\infty})$
is a frame of $S$ with dual $(-\hat{\sigma}_{\infty},
d_{X_{1}}\sigma_{\infty},d_{X_{2}}\sigma_{\infty},-\sigma_{\infty})$
and we conclude that
$$|\mathcal{N}|^{2}=2\sum_{i,j}(\mathcal{N}_{X_{i}}(d_{X_{j}}\sigma_{\infty}),\mathcal{N}_{X_{i}}(d_{X_{j}}\sigma_{\infty})).$$
Lemma \ref{NperpisoSperp} establishes $\mathcal{N}_{X_{i}}(d_{X_{j}}\sigma_{\infty})=\mathcal{Q}(\Pi
_{\infty}(X_{i},X_{j}))$ and completes the proof.
\end{proof}

\subsubsection{Willmore surfaces and harmonicity}

A conformal immersion is \textit{Willmore} it it extremizes the Willmore functional and \textit{constrained Willmore} it it extremizes the Willmore functional with respect to variations that infinitesimally preserve the conformal structure, that is, variations satisfying   
$$\frac{d}{d t}_{\vert{t=0}}(\delta_{z},\delta_{z})_{t}=0,$$ for the
variation $(,)_{t}$ of the induced metric and $z$ a holomorphic chart.

Theorem \ref{Wenergy} makes it clear that
$$\mathcal{W}(\Lambda)=E(S,\mathcal{C}_{\Lambda}),$$ the Willmore energy of $\Lambda$
coincides with the Dirichlet energy of S with respect to any of the
metrics in the conformal class $\mathcal{C}_{\Lambda}$ (although the
Levi-Civita connection is not conformally-invariant, the Dirichlet
energy of a mapping of a surface is preserved under conformal
changes of the metric (and so is its harmonicity)). Furthermore, in
a very well known result, established by Blaschke \cite{blaschke}, for $n=3$, and, independently, by Ejiri \cite{ejiri} and Rigoli
\cite{rigoli}, for general $n$:
\begin{thm}\cite{blaschke, ejiri,rigoli}\label{LWiffSharmonic} $\Lambda$ is a Willmore surface if and
only if its central sphere congruence
$S:(\Sigma,\mathcal{C}_{\Lambda})\rightarrow\mathcal{G}$ is a
harmonic map.
\end{thm}

Next we present a proof of Theorem \ref{LWiffSharmonic} in the
light-cone picture. This is a generalization of the proof presented
in \cite{BFLPP}, in the quaternionic setting, for the particular
case of $n=4$.\newline

\begin{proof} Given a variation $\Lambda_{t}$ of $\Lambda$ and
$S_{t}$ the corresponding variation of $S$ through central sphere
congruences, the Dirichlet energy $E(S_{t},\mathcal {C}_{t})$ of
$S_{t}$ with respect to the conformal structure $\mathcal {C}_{t}$
induced in $\Sigma$ by $\Lambda_{t}$ is given by
$\frac{1}{2}\int_{\Sigma}(dS_{t}\wedge *_{t}dS_{t})$, for
$\mathrm{dA}_{t}$ and $*_{t}$ the area element and the Hodge
$*$-operator of $(\Sigma,g_{t})$, respectively, fixing
$g_{t}\in\mathcal{C}_{t}$. Hence
$$\frac{d}{dt}_{\mid_{t=0}}E(S_{t},\mathcal {C}_{t})=\frac{1}{2}\int
_{\Sigma}((d\dot{S}\wedge
*dS)+(dS\wedge\dot{*}dS)+(dS\wedge*d\dot{S})),$$ abbreviating
$\frac{d}{dt}_{\mid_{t=0}}$ by a dot. Let $(J_{t})_{t}$ be the
corresponding variation of $J$ through canonical complex structures.
Differentiation at $t=0$ of $*_{t}dS_{t}=-(dS_{t})J_{t}$ gives
$\dot{*}dS=-(dS)\dot{J}$, whilst that of $J_{t}^{2}=-I$ gives
$\dot{J}J=-J\dot{J}$ and, in particular, that $\dot{J}$ intertwines
the eigenspaces of $J$. The $\mathcal{C}_{\Lambda}$-conformality of
$S$, $(d_{X\pm iJX}S,d_{X\pm iJX}S)=0$, respectively, for
$X\in\Gamma(T\Sigma)$, establishes then $(dS\wedge\dot{*}dS)=0$ and,
therefore,
$$\frac{d}{dt}_{\mid_{t=0}}E(S_{t},\mathcal
{C}_{t})=\frac{d}{dt}_{\mid_{t=0}}E(S_{t},\mathcal {C}_{0}).$$ It is
now clear that if
$S:(\Sigma,\mathcal{C}_{\Lambda})\rightarrow\mathcal{G}$ is harmonic
then $\Lambda$ is Willmore.

Conversely, suppose $\Lambda$ is Willmore, fix $z$ a holomorphic
chart of $(\Sigma,\mathcal{C}_{\Lambda})$ and let us show that the
tension field $\tau_{z}$ of $S:(\Sigma,g_{z})\rightarrow
\mathcal{G}$ vanishes. First of all, observe that
$$4\nabla_{\delta_{z}}S_{\bar{z}}=\tau_{z}=4\nabla_{\delta_{\bar{z}}}S_{z}$$
to conclude that $\Lambda^{(1)}\subset \mathrm{ker}\,\tau_{z}$: by
\eqref{Hillseq},
$$(\nabla_{\delta_{z}}S_{\bar{z}})\sigma_{z}^{z}=\nabla^{S^{\perp}}_{\delta_{z}}(\pi_{S^{\perp}}\sigma_{z\bar{z}}^{z})-\pi_{S^{\perp}}(\nabla^{S}_{\delta_{z}}\sigma_{z}^{z})_{\bar{z}}=0,$$
and, similarly,
$(\nabla_{\delta_{\bar{z}}}S_{z})\sigma_{\bar{z}}^{z}=0$, whilst
$(\nabla_{\delta_{z}}S_{\bar{z}})\sigma^{z}=0$ is immediate. It
follows that $\mathrm{Im}\tau_{z}^{t}\subset\Lambda$. Fix
$\Lambda_{t}=\langle\sigma_{t}\rangle$ a variation of $\Lambda$ and
let $S_{t}$ be the corresponding variation of $S$ through central
sphere congruences. Then
$\tau_{z}^{t}(\pi_{S^{\perp}}\dot{\sigma})=\lambda\sigma_{0}$, for
some $\lambda\in C^{\infty}(\Sigma,\R)$, and, therefore,
$\mathrm{tr}(\tau_{z}^{t}\dot{S})=\lambda$ (note that
$\dot{S}\sigma=\pi_{S^{\perp}}\dot{\sigma}$). On the other hand,
classically, $$0=\frac{d}{dt}_{\mid_{t=0}}E(S_{t},\mathcal
{C}_{\Lambda})=-\int_{\Sigma}(\dot{S},\tau_{z})dA_{z}=-\int_{\Sigma}\mathrm{tr}(\tau_{z}^{t}\dot{S})dA_{z},$$
for $dA_{z}$ the area element of $(\Sigma,g_{z})$. Now suppose
$\tau_{z}$ is non-zero. Then so is
$\tau_{z}^{t}\in\Gamma(\mathrm{Hom}(S^{\perp},S))$, so we can choose
$\sigma_{t}$ such that $\lambda$ is positive, which leads to a
contradiction and completes the proof.
\end{proof}

Having characterized Willmore surfaces by the harmonicity of the
central sphere congruence, and recalling \eqref{eq:dSmathcalN}, we
deduce the Willmore surface equation,
\begin{equation}\label{eq:Willmore eq}
d^{\mathcal{D}}*\mathcal{\mathcal{N}}=0.
\end{equation}
More generally, we have a manifestly conformally-invariant
characterization of constrained Willmore surfaces in space-forms,
first established in \cite{SD} and reformulated in
\cite{burstall+calderbank} as follows:

\begin{thm}\cite{burstall+calderbank} \label{CWtemp}
$\Lambda$ is a constrained Willmore surface if and only if there
exists a real form $q\in\Omega ^{1}(\Lambda\wedge\Lambda ^{(1)})$
with
\begin{equation}\label{eq:curlyDextderivofq}
d^{\mathcal{D}}q=0
\end{equation}
such that
\begin{equation}\label{eq:mainCWeq}
d^{\mathcal{D}}*\mathcal{N}=2\,[q\wedge *\mathcal{N}].
\end{equation}
Such a form $q$ is said to be a \emph{[Lagrange] multiplier} for $\Lambda$ and $\Lambda$ is said to be a \emph{$q$-constrained Willmore surface}. 
\end{thm}

\begin{proof}
Calculus of variations techniques show that
the variational Willmore energy relates to the variational surface
by
$$\dot{\mathcal{W}}=\int_{\Sigma}((d^{\mathcal{D}}*\mathcal{\mathcal{N}})\wedge\dot{\Lambda}),$$
for some non-degenerate pairing $(\,\wedge\,)$. For general
variations, $\dot{\Lambda}$ can be arbitrary, establishing the
Willmore surface equation \eqref{eq:Willmore eq}, whereas
infinitesimally conformal variations are characterized by the normal
variational being in the image of the conformal Killing operator
$\overline{\bar{\partial}}$, which, according to Weyl's Lemma,
consists of $(H^{0}K)^{\perp}$, for $H^{0}K$ the space of
holomorphic quadratic differentials. The result follows by defining
a multiplier $q$ from a quadratic differential $q^{z}dz^{2}\in
(H^{0}K)^{\perp}$ via
$q_{\delta_{z}}\sigma_{z}=-\frac{1}{2}\,q^{z}\sigma$, for $\sigma$ a
lift of $\Lambda$ and $z$ a holomorphic chart of $\Sigma$.
\end{proof}

Willmore surfaces are the $0$-constrained Willmore surfaces. The
zero multiplier is not necessarily the only multiplier for a
constrained Willmore surface with no constraint on the conformal
structure, though. In fact, the uniqueness of multiplier
characterizes non-isothermic constrained Willmore surfaces, as we
shall see below in this text.

The characterization of constrained Willmore surfaces above
motivates a natural extension to surfaces that are not necessarily
compact.

Next we present a useful result, which establishes, in particular,
that if $q$ is a multiplier for $\Lambda$, then $q^{1,0}$ takes
values in $\Lambda \wedge\Lambda ^{0,1}$.

\begin{lemma}\label{withvswithoutdecomps}
Given $q\in\Omega ^{1}(\Lambda\wedge\Lambda ^{(1)})$ real,

$i)$\,\,if $d^{\mathcal{D}}q=0$ then $q^{1,0}\in\Omega
^{1,0}(\Lambda \wedge\Lambda ^{0,1})$ or, equivalently,
$q^{0,1}\in\Omega ^{0,1}(\Lambda \wedge\Lambda ^{1,0})$;

$ii)$\,\,$d^{\mathcal{D}}q=0$ if and only if
$d^{\mathcal{D}}q^{1,0}=0$, or, equivalently,
$d^{\mathcal{D}}q^{0,1}=0;$

$iii)$\,\,$d^{\mathcal{D}}q=0$ if and only if $d^{\mathcal{D}}*q=0$.
\end{lemma}

\begin{proof}
Fix $z$ a holomorphic chart of $\Sigma$. First of all, observe that
a section $\xi$ of $\Lambda \wedge\Lambda ^{(1)}$ is a section of
$\Lambda \wedge \Lambda ^{1,0}$ if and only if $\xi(\sigma
_{z}^{z})=0$. Suppose $d^{\mathcal{D}}q=0$. Then, in particular,
$d^{\mathcal{D}}q\,(\delta_{z},\delta_{\bar{z}})\,\sigma^{z}=0$, or,
equivalently, $$\mathcal{D}_{\delta_{z}}(q_{\delta
_{\bar{z}}}\sigma^{z})-q_{\delta_{\bar{z}}}(\mathcal{D}_{\delta_{z}}\sigma^{z})-\mathcal{D}_{\delta_{\bar{z}}}(q_{\delta_{z}}\sigma^{z})+q_{\delta_{z}}(\mathcal{D}_{\delta_{\bar{z}}}\sigma^{z})=0,$$
establishing
\begin{equation}\label{eq:qsigmazdeltabarzigualqsigmabarzdeltaz}
q_{\delta_{z}}\sigma^{z}_{\bar{z}}=q_{\delta_{\bar{z}}}\sigma_{z}^{z}.
\end{equation}
In its turn,
$d^{\mathcal{D}}q\,(\delta_{z},\delta_{\bar{z}})\,\sigma_{z}^{z}=0$ implies $$\mathcal{D}_{\delta_{z}}(q_{\delta
_{\bar{z}}}\sigma_{z}^{z})-\mathcal{D}_{\delta_{\bar{z}}}(q_{\delta_{z}}\sigma_{z}^{z})+q_{\delta_{z}}
\sigma_{z\bar{z}}^{z}=0,$$ by \eqref{Hillseq}. On the other hand, the
skew-symmetry of $q$ establishes
$(q\sigma_{z\bar{z}}^{z},\sigma_{z\bar{z}}^{z})=0$ and, therefore,
\begin{equation}\label{eq:que}
q\sigma_{z\bar{z}}^{z}=\mu\sigma_{z}+\eta\sigma_{\bar{z}}^{z},
\end{equation}
for some $\mu,\eta\in\Omega^{1}(\underline{\C})$. Hence
$$\mathcal{D}_{\delta_{z}}(q_{\delta
_{\bar{z}}}\sigma_{z}^{z})+\mu_{\delta
_{z}}\sigma_{z}^{z}=\mathcal{D}_{\delta_{\bar{z}}}(q_{\delta_{z}}
\sigma_{z}^{z})-\eta_{\delta_{z}}\sigma_{\bar{z}}^{z}.$$ It is
obvious that a section of $\Lambda\wedge\Lambda ^{(1)}$ transforms
sections of $\Lambda^{(1)}$ into sections of $\Lambda$, so that, in
particular, both $q_{\delta _{\bar{z}}}\sigma_{z}^{z}$ and
$q_{\delta_{z}} \sigma_{z}^{z}$ are sections of $\Lambda$. We
conclude that $\mathcal{D}_{\delta_{z}}(q_{\delta
_{\bar{z}}}\sigma_{z}^{z})+\mu_{\delta _{z}}\sigma_{z}^{z}$ is a
section of $\Lambda^{1,0}\cap\Lambda^{0,1}=\Lambda$. Write
$q_{\delta _{\bar{z}}}\sigma_{z}^{z}=\lambda\sigma^{z}$, with
$\lambda\in\Gamma(\underline{\C})$. Then $$\lambda_{z}\sigma^{z}
+(\lambda +\mu_{\delta _{z}})\sigma_{z}^{z}=\gamma\sigma^{z},$$ for
some $\gamma\in\Gamma (\underline{\C})$. In particular,
$\lambda=-\mu_{\delta _{z}}$. Equation \eqref{eq:que} establishes,
on the other hand, $$q=-2\mu\,
\sigma^{z}\wedge\sigma^{z}_{z}-2\eta\,\sigma^{z}\wedge\sigma^{z}_{\bar{z}}$$
and, in particular,
$q_{\delta_{z}}\sigma^{z}_{\bar{z}}=\mu_{\delta_{z}}\sigma^{z}$.
Equation \eqref{eq:qsigmazdeltabarzigualqsigmabarzdeltaz} completes
the proof of $i)$.

Next we prove $ii)$. By \eqref{Hillseq},
$$\mathcal{D}^{1,0}\Gamma(\Lambda^{1,0})\subset\Omega^{1,0}(\Lambda^{1,0})$$
or, equivalently,
$$\mathcal{D}^{0,1}\Gamma(\Lambda^{0,1})\subset\Omega^{0,1}(\Lambda^{0,1})$$
and, therefore, following $i)$, $d^{\mathcal{D}}q^{1,0}\in\Omega
^{2}(\Lambda\wedge\Lambda^{0,1})$ and
$d^{\mathcal{D}}q^{0,1}\in\Omega ^{2}(\Lambda \wedge\Lambda
^{1,0})$. Hence $d^{\mathcal{D}}q=0$ forces $d^{\mathcal{D}}q^{1,0}$
and $d^{\mathcal{D}}q^{0,1}$ to vanish separately. The reality of
$q$ completes the proof of $ii)$.

As for $iii)$, it is immediate from $ii)$.
\end{proof}

\subsubsection{Constrained Willmore surfaces: a zero-curvature characterization}

For a map into a Grassmannian, harmonicity amounts to the flatness of a family of connections, according to Uhlenbeck \cite{uhlenbeck}. With the characterization of Willmore surfaces by the harmonicity of
the central sphere congruence, a zero-curvature characterization of Wilmore surfaces  follows. More generally, the constrained Willmore surface equations amount to the flatness of a certain\footnote{The associated family of flat connections presented in \cite{BQ} corresponds to a different choice of orientation in $\Sigma$.} family of connections too:
\begin{thm}\cite{burstall+calderbank} \label{CWzerocurv}
$\Lambda$ is a constrained Willmore surface if and only if there
exists a real form $q\in\Omega^{1}(\Lambda\wedge\Lambda^{(1)})$ such that 
$$d^{\lambda}_{q}:=\mathcal{D}+\lambda
^{-1}\mathcal{N}^{1,0}+\lambda\mathcal{N}^{0,1}+(\lambda
^{-2}-1)q^{1,0}+(\lambda ^{2}-1)q^{0,1}$$ is flat for all
$\lambda\in S^{1}$.
\end{thm}

Before proceeding to the proof of the theorem, and for further reference, observe that, given $a,b\in\R^{n+1,1}$ and $T\in
o(\mathbb{R}^{n+1,1})$, $$[T,a\wedge b]=(Ta)\wedge b+a\wedge (Tb),$$
to conclude that
\begin{equation}\label{eq:LieLambdawedgeLambda1}
[\Lambda\wedge\Lambda^{(1)},\Lambda\wedge\Lambda^{(1)}]\subset\Lambda\wedge\Lambda=\{0\}.
\end{equation}
Now we proceed to the proof of the theorem:
\begin{proof}
According to the decomposition
\begin{equation}\label{eq:SSperpparallelornot}
o(\underline{\R}^{n+1,1})=(\wedge^{2}S\oplus\wedge^{2}S^{\perp})\oplus
S\wedge S^{\perp},
\end{equation}
the flatness of $d$, characterized by
$$0=R^{\mathcal{D}}+d^{\mathcal{D}}\mathcal{N}+\frac{1}{2}\,[\mathcal{N}\wedge
\mathcal{N}],$$ encodes two structure equations,
namely, 
\begin{equation}\label{eq:GReq}
R^{\mathcal{D}}+\frac{1}{2}\,[\mathcal{N}\wedge
\mathcal{N}]=0
\end{equation}
and 
\begin{equation}\label{eq:Ceq}
d^{\mathcal{D}}\mathcal{N}=0.
\end{equation}
Now suppose $q\in\Omega^{1}(\Lambda\wedge\Lambda^{(1)})$ is a real form. Given $\lambda\in S^{1}$, set $$A^{\lambda}=d^{\lambda}_{q}-\mathcal{D}\in\Omega^{1}(\mathrm{End}(\underline{R}^{n+1,1})).$$ The curvature tensor of $d^{\lambda}_{q}$ is given by 
$$R^{d^\lambda_{q}}=R^{\mathcal{D}}+d^{\mathcal{D}}A^{\lambda}+\frac{1}{2}[A^{\lambda}\wedge A^{\lambda}].$$Since there are no non-zero
$(2,0)$- or $(0,2)$-forms over a surface, we have 
$$\frac{1}{2}[A^{\lambda}\wedge A
^{\lambda}]=[\mathcal{N}^{1,0}\wedge
\mathcal{N}^{0,1}]+(\lambda
^{-1}-\lambda)\big([q^{1,0}\wedge \mathcal{N}^{0,1} ]-[
q^{0,1}\wedge\mathcal{N}^{1,0}]\big)+(2-\lambda
^{-2}-\lambda ^{2})[q^{1,0}\wedge q^{0,1}]$$ and equation \eqref{eq:GReq} 
establishes then
$$R^{d^{\lambda}_{q}}=d^{\mathcal{D}} A^{\lambda}+(\lambda
^{-1}-\lambda)\big([q^{1,0}\wedge\mathcal{N}^{0,1}
]-[q^{0,1}\wedge\mathcal{N}^{1,0}
])+\frac{1}{2}\,(2-\lambda ^{-2}-\lambda ^{2})[q\wedge q].$$But, according to   \eqref{eq:LieLambdawedgeLambda1}, $[q\wedge q]=0$. Hence
$$R^{d^{\lambda}_{q}}=d^{\mathcal{D}} A^{\lambda}+(\lambda
^{-1}-\lambda)\big([q^{1,0}\wedge\mathcal{N}^{0,1}
]-[q^{0,1}\wedge\mathcal{N}^{1,0}
]).$$
In its
turn, equation \eqref{eq:Ceq} gives
$$d^{\mathcal{D}}\mathcal{N}^{1,0}=\frac{i}{2}\,d^{\mathcal{D}}*\mathcal{N}=-d^{\mathcal{D}}\mathcal{N}^{0,1}.$$
We conclude that
$$R^{d^{\lambda}_{q}}=\frac{\lambda
^{-1}-\lambda}{2}\,i\,(d^{\mathcal{D}}*\mathcal{N}-2[q\wedge
*\mathcal{N}])+(\lambda
^{-2}-1)\,d^{\mathcal{D}}q^{1,0}+(\lambda
^{2}-1)\,d^{\mathcal{D}}q^{0,1}.$$
Yet again, according to the decomposition \eqref{eq:SSperpparallelornot}, 
it follows that
$R^{d^{\lambda}_{q}}=0$ if and only if both
\begin{equation}\label{eq:CWeqlambda1}
\frac{\lambda
^{-1}-\lambda}{2}\,i\,(d^{\mathcal{D}}*\mathcal{N}-2[q\wedge
*\mathcal{N}])=0
\end{equation}
and
\begin{equation}\label{eq:CWeqlambda2}
(\lambda ^{-2}-1)\,d^{\mathcal{D}}q^{1,0}+(\lambda
^{2}-1)\,d^{\mathcal{D}}q^{0,1}=0
\end{equation}
hold. Organizing equations \eqref{eq:CWeqlambda1} and
\eqref{eq:CWeqlambda2} by powers of $\lambda$ completes the proof.
\end{proof}

\subsubsection{Constrained Willmore surfaces and the isothermic surface condition}

Isothermic surfaces are classically defined by the existence of
conformal curvature line coordinates. Equation \eqref{eq:Pichanges} makes it clear that, although the second fundamental form is not conformally invariant, conformal curvature
line coordinates are preserved under conformal changes of the metric
and, therefore, so is the isothermic surface condition. The next
result presents a manifestly conformally-invariant formulation of
the isothermic surface condition, established in \cite{BDPP} and
discussed also in \cite{BS}.

\begin{lemma}\cite{BDPP,BS}
$\Lambda$ is isothermic if and only if there exists a non-zero
closed real $1$-form
$\eta\in\Omega^{1}(\Lambda\wedge\Lambda^{(1)})$. Under these
conditions, we may say that $(\Lambda,\eta)$ is an isothermic
surface. The form $\eta$ is defined up to a real constant scale.
\end{lemma}

\begin{rem}\label{deta0}
According to the decomposition \eqref{eq:SSperpparallelornot}, given
$\eta\in\Omega^{1}(\Lambda\wedge\Lambda^{(1)})$,
$d\eta=d^{\mathcal{D}}\eta+[\mathcal{N}\wedge\eta]$ vanishes if and
only if $d^{\mathcal{D}}\eta=0=[\mathcal{N}\wedge\eta]$.
\end{rem}

\begin{prop}\cite{BQ}\label{uniqq}
A constrained Willmore surface has a unique multiplier if and only
if it is not an isothermic surface. Furthermore:

$i)$\,\,if $q_{1}\neq q_{2}$ are multipliers for $\Lambda$, then
$(\Lambda,*(q_{1}-q_{2}))$ is isothermic;

$ii)$\,\,if $(\Lambda,\eta)$ is an isothermic $q$-constrained
Willmore surface, then the set of multipliers to $\Lambda$ is the
affine space $q+\langle *\eta\rangle_{\R}$.
\end{prop}

\begin{proof}
It is immediate, noting that $[\mathcal{N}\wedge\eta]=[*\eta\wedge
*\mathcal{N}]$ and recalling Lemma \ref{withvswithoutdecomps} -
$iii)$.
\end{proof}

A classical result by Thomsen \cite{thomsen} characterizes
isothermic Willmore surfaces in $3$-space as minimal surfaces in
some $3$-dimensional space-form. Constant mean curvature surfaces in
$3$-dimensional space-forms are examples of isothermic constrained
Willmore surfaces, as proven by J. Richter \cite{richter}. However,
isothermic constrained Willmore surfaces in $3$-space are not
necessarily constant mean curvature surfaces in some space-form, as
established by an example, presented in \cite{BPP}, of a constrained
Willmore cylinder that does not have constant mean curvature in any
space-form.

\subsection{Complexified constrained Willmore surfaces}

The transformations of a constrained Willmore surface $\Lambda$ we
present below in this work are, in particular, pairs
$((\Lambda^{1,0})^{*},(\Lambda^{0,1})^{*})$ of transformations
$(\Lambda^{1,0})^{*}$ and $(\Lambda^{0,1})^{*}$ of $\Lambda^{1,0}$
and $\Lambda^{0,1}$, respectively. The fact that $\Lambda^{1,0}$ and
$\Lambda^{0,1}$ intersect in a rank $1$ bundle will ensure that
$(\Lambda^{1,0})^{*}$ and $(\Lambda^{0,1})^{*}$ have the same
property. The isotropy of $\Lambda^{1,0}$ and $\Lambda^{0,1}$ will
ensure that of $(\Lambda^{1,0})^{*}$ and $(\Lambda^{0,1})^{*}$ and,
therefore, that of their intersection. The reality of the bundle
$\Lambda^{1,0}\cap\Lambda^{0,1}$ is preserved by the spectral
deformation, but it is not clear that the same is necessarily true
for B\"{a}cklund transformation. This motivates the definition of
\textit{complexified surface}.

Fix a conformal structure $\mathcal{C}$ on $\Sigma$ and consider the
corresponding complex structure on $\Sigma$. Let $\hat{d}$ be a flat
metric connection on $\underline{\C}^{n+2}$ and $d$ denote the
trivial flat connection. In what follows, omitting the reference to
some specific connection shall be understood as an implicit
reference to $d$.

\begin{defn}\label{defcomplexsurface}
We define a \emph{complexified} $\hat{d}$-\emph{surface} to be a
pair $(\Lambda^{1,0},\Lambda^{0,1})$ of isotropic rank $2$
subbundles of $\underline{\C}^{n+2}$ intersecting in a rank $1$
bundle $$\Lambda:=\Lambda^{1,0}\cap\Lambda^{0,1}$$ such that
$$\hat{d}^{1,0}\Gamma\Lambda\subset\Omega^{1,0}\Lambda^{1,0},\,\,\,\,\,\,\hat{d}^{0,1}\Gamma\Lambda\subset\Omega^{0,1}\Lambda^{0,1}.$$
\end{defn}
Obviously, given $\Lambda$ a (real) surface in
$\mathbb{P}(\mathcal{L})$, $(\Lambda^{1,0}, \Lambda^{0,1})$ is a
complexified surface with respect to $\mathcal{C}_{\Lambda}$.
Henceforth, we drop the term "complexified" and use \textit{real
surface} when referring explicitly to a complexified surface
$(\Lambda^{1,0}, \Lambda^{0,1})$ defining a real surface $\Lambda$.
Observe that $(\Lambda^{1,0},\Lambda^{0,1})$ is a real surface if
and only if $\Lambda$ is a real bundle (recall that $\Lambda$ is an
immersion if and only if the bundle $\Lambda^{1,0}+\Lambda^{0,1}$
has rank $3$).

Observe, on the other hand, that, in the particular case of a real
surface $(\Lambda^{1,0},\Lambda^{0,1})$, the notation
$\Lambda^{1,0}$ and $\Lambda^{0,1}$ is consistent with
\eqref{eq:Lambdaij}. Indeed, the $\mathcal{C}$-isotropy of
$\Lambda^{1,0}$ characterizes the $\mathcal{C}$-conformality of the
lifts of $\Lambda$, or equivalently, the fact that
$\mathcal{C}=\mathcal{C}_{\Lambda}$.

\subsubsection{Constrained Willmore surfaces and perturbed 
harmonic bundles}

Theorem \ref{CWzerocurv} motivates the definition of \textit{perturbed 
harmonicity} for a bundle, which we present next and which will apply
to the central sphere congruence to provide a characterization of constrained Willmore surfaces. In the particular case of a bundle of $(3,1)$-planes in $\R^{n+1,1}$, our notion of perturbed harmonicity coincides with the notion of $2$-\emph{perturbed harmonicity}  introduced in \cite{BQ}.

Given $V$ a non-degenerate subbundle of $\underline{\C}^{n+2}$,
consider the decomposition
$\hat{d}=\mathcal{D}_{V}^{\hat{d}}+\mathcal{N}_{V}^{\hat{d}}$ for
$\mathcal{D}_{V}^{\hat{d}}$ the metric connection on
$\underline{\C}^{n+2}$ given by the sum of the connections induced
on $V$ and $V^{\perp}$ by $\hat{d}$.

\begin{defn}
A non-degenerate rank $\,4$ subbundle $V$ of $\underline{\C}^{n+2}$
is said to be a central sphere congruence of a
$\hat{d}$-surface $(\Lambda^{1,0},\Lambda^{0,1})$ if
$$\Lambda^{1,0}+\Lambda^{0,1}\subset V$$ and
$$(\mathcal{N}_{V}^{\hat{d}})^{1,0}\Lambda^{0,1}=0=(\mathcal{N}_{V}^{\hat{d}})^{0,1}\Lambda^{1,0}.$$
\end{defn}

\begin{rem}
Let $(\Lambda^{1,0},\Lambda^{0,1})$ be a surface, $\sigma\neq 0$ be
a section of $\Lambda$ and $z$ be a holomorphic chart of $\Sigma$.
Note that $\Lambda^{1,0}+\Lambda^{0,1}\subset V$ establishes
$$(\mathcal{N}_{V})_{\vert\Lambda}=0$$ and then
$\mathcal{N}_{V}^{1,0}\Lambda^{0,1}=0=\mathcal{N}_{V}^{0,1}\Lambda^{1,0}$
reads $\sigma_{z\bar{z}}\in\Gamma V$. Hence, generically (if
$\sigma\wedge\sigma_{z}\neq 0\neq\sigma\wedge\sigma_{\bar{z}}$),
$\Lambda^{1,0}$, $\Lambda^{0,1}$ and $V$ are all determined by
$\Lambda$: $\Lambda^{1,0}=\langle\sigma,\sigma_{z}\rangle$,
$\Lambda^{0,1}=\langle\sigma,\sigma_{\bar{z}}\rangle$ and
$$V=\langle\sigma,\sigma_{z},\sigma_{\bar{z}},\sigma_{z\bar{z}}\rangle.$$
In particular, the complexification of the central sphere congruence
of a real surface $\Lambda$ is the unique central sphere congruence
of the corresponding surface $(\Lambda^{1,0},\Lambda^{0,1})$.
\end{rem}

For further reference:

\begin{lemma}\label{curlyDfacts}
Suppose $V$ is a central sphere congruence of a
$\hat{d}$-surface $(\Lambda^{1,0},\Lambda^{0,1})$. Then
$$(\mathcal{D}_{V}^{\hat{d}})^{1,0}\Gamma\Lambda^{1,0}\subset\Omega^{1,0}\Lambda^{1,0},\,\,\,\,(\mathcal{D}_{V}^{\hat{d}})^{0,1}\Gamma\Lambda^{0,1}\subset\Omega^{0,1}\Lambda^{0,1}.$$
\end{lemma}
\begin{proof}
First of all, observe that, as
$\mathrm{rank}\,\Lambda^{1,0}=\frac{1}{2}\,\mathrm{rank}\,V=\mathrm{rank}\,\Lambda^{0,1}$,
the isotropy of both $\Lambda^{1,0}$ and $\Lambda^{0,1}$ establishes
their maximal isotropy in $V$. Write $\Lambda^{1,0}=\langle\sigma,\tau\rangle$, with
$\sigma\in\Gamma\Lambda$. Since 
$$(\mathcal{D}_{V}^{\hat{d}})^{1,0}\sigma=\pi_{V}\circ\hat{d}^{1,0}\circ\pi_{V}\sigma\in\Omega^{1,0}\Lambda^{1,0},$$
the fact that
$\mathcal{D}_{V}^{\hat{d}}$ is a metric connection, together with the isotropy of $\Lambda^{1,0}$, shows that
$$((\mathcal{D}_{V}^{\hat{d}})^{1,0}\tau,\sigma)=-(\tau,
(\mathcal{D}_{V}^{\hat{d}})^{1,0}\sigma)=0,$$ whereas  
$$((\mathcal{D}_{V}^{\hat{d}})^{1,0}\tau,\tau)=\frac{1}{2}\,d^{1,0}(\tau,\tau)=0.$$
We conclude that
$(\mathcal{D}_{V}^{\hat{d}})^{1,0}\tau\perp\Lambda^{1,0}$ and,
therefore, that $(\mathcal{D}_{V}^{\hat{d}})^{1,0}\tau$ takes values in $\Lambda^{1,0}$. A similar
argument establishes
$(\mathcal{D}_{V}^{\hat{d}})^{0,1}\Gamma\Lambda^{0,1}\subset\Omega^{0,1}\Lambda^{0,1}$.
\end{proof}

\begin{defn}
A non-degenerate bundle $V\subset\underline{\C}^{n+2}$ is said to be
$\hat{d}$-\emph{perturbed harmonic} if there exists a $1$-form $q$
with values in $\wedge^{2}V\oplus \wedge^{2}V^{\perp}$ such that,
for each $\lambda\in\C\backslash \{0\}$, the metric connection
$$\hat{d}^{\lambda ,q}_{V}:=\mathcal{D}_{V}^{\hat{d}}+\lambda
^{-1}(\mathcal{N}_{V}^{\hat{d}})^{1,0}+\lambda(\mathcal{N}_{V}^{\hat{d}})^{0,1}+(\lambda
^{-2}-1)q^{1,0}+(\lambda ^{2}-1)q^{0,1},$$ on
$\underline{\C}^{n+2}$, is flat. In this case, we say that $V$ is
\emph{$(q,\hat{d})$-perturbed harmonic} or, in the case $q=0$,
\emph{$\hat{d}$-harmonic}. In the particular case of $\hat{d}=d$,
$V$ a real bundle and $q$ a real form, we say that $V$ is a
\emph{real $q$-perturbed harmonic bundle}.
\end{defn}

\begin{defn}
A $\hat{d}$-surface $(\Lambda^{1,0},\Lambda^{0,1})$ is said to be
 \emph{$\hat{d}$-constrained Willmore} if it admits a $(q,\hat{d})$-perturbed 
harmonic central sphere congruence with
\begin{equation}\label{eq:multiplierspecifics}
q^{1,0}\in\Omega^{1,0}(\wedge^{2}\Lambda^{0,1}),\,\,\,\,\,\,q^{0,1}\in\Omega^{0,1}(\wedge^{2}\Lambda^{1,0}).
\end{equation}
If $q=0$, we say that $(\Lambda^{1,0},\Lambda^{0,1})$ is
\emph{$\hat{d}$-Willmore}. In the particular case of $\hat{d}=d$,
$(\Lambda^{1,0},\Lambda^{0,1})$ a real surface and $q$ a real form,
we say that $(\Lambda^{1,0},\Lambda^{0,1})$ is a \emph{real
$q$-constrained Willmore surface}.
\end{defn}

From Theorem \ref{CWzerocurv} and Lemma 
\ref{withvswithoutdecomps}, it follows that the real constrained
Willmore surface condition is preserved under the correspondence
$$(\Lambda^{1,0},\Lambda^{0,1})\longleftrightarrow\Lambda^{1,0}\cap\Lambda^{0,1}$$
for real surfaces $(\Lambda^{1,0},\Lambda^{0,1})$, with preservation
of multipliers:
\begin{thm}\cite{BQ}\label{CWzerocurvature}
Suppose $(\Lambda^{1,0},\Lambda^{0,1})$ is a real surface. Then
$\Lambda$ is constrained Willmore if and only if $S$ is
$q$-perturbed harmonic, for some real $1$-form $q$ with
$q^{1,0}\in\Omega^{1,0}(\wedge^{2}\Lambda^{0,1})$.
\end{thm}

\section{Transformations of perturbed harmonic bundles and constrained Willmore surfaces}

Fix a conformal structure $\mathcal{C}$ on $\Sigma$ and consider the
corresponding complex structure on $\Sigma$. Let $\mathrm{Ad}$
denote the adjoint representation of the orthogonal group on the
orthogonal algebra. Note that, given $T\in O(\mathbb{R}^{n+1,1})$
and $u,v\in\R^{n+1,1}$, $$\mathrm{Ad}_{T}(u\wedge v)=Tu\wedge Tv.$$

Let $V$ be a non-degenerate subbundle of $\underline{\C}^{n+2}$ and
$\pi_{V}$ and $\pi_{V^{\perp}}$ denote the orthogonal projections of
$\underline{\C}^{n+2}$ onto $V$ and $V^{\perp}$, respectively, and
$\rho$ denote reflection across $V$, $$\rho=\pi_{V}-\pi_{V^{\perp}}.$$
Let $(\Lambda^{1,0},\Lambda^{0,1})$ be a surface admitting $V$ as a
central sphere congruence. As usual, we write $\Lambda$ for
$\Lambda^{1,0}\cap\Lambda^{0,1}$. Suppose $V$ is $q$-perturbed 
harmonic for some $q\in\Omega^{1}(\wedge^{2}V\oplus
\wedge^{2}V^{\perp})$ satisfying conditions
\eqref{eq:multiplierspecifics}.

\subsection{Spectral deformation}

For each $\lambda\in\C\backslash\{0\}$, the flatness of the metric
connection $d^{\lambda,q}_{V}$ on $\underline{\mathbb C}^{n+2}$
establishes the existence of an isometry
$$\phi^{\lambda,q}_{V}:(\underline{\mathbb C}^{n+2},
d^{\lambda,q}_{V})\rightarrow(\underline{\mathbb C}^{n+2},d)$$ of
bundles, preserving connections, defined on a simply connected
component of $\Sigma$ and unique up to a M\"{o}bius transformation.

\begin{lemma}\label{changeconnection} Let $\hat{d}$ be a flat metric connection
on $\underline{\C}^{n+2}$ and $$\phi:(\underline{\mathbb C}^{n+2},
\hat{d})\rightarrow(\underline{\mathbb C}^{n+2},d)$$ be an isometry
of bundles, preserving connections. Then $V$ is
$(q,\hat{d})$-perturbed harmonic, for some $q$, if and only if
$\phi V$ is $\mathrm{Ad}_{\phi}q$-perturbed harmonic.
\end{lemma}
\begin{proof}
It is immediate from the fact that
\begin{equation}\label{eq:curlyNphi}
\mathcal{D}_{\phi
V}=\phi\circ\mathcal{D}_{V}^{\hat{d}}\circ\phi^{-1},\,\,\,\,\mathcal{N}_{\phi
V}=\phi\,\mathcal{N}_{V}^{\hat{d}}\,\phi^{-1}
\end{equation}
and, therefore, $$d_{\phi V}^{\lambda,q}=\phi\circ
\hat{d}\,^{\lambda,\mathrm{Ad}_{\phi^{-1}}q}_{V}\circ\phi^{-1}.$$
\end{proof}

Set
$$q_{\lambda}:=\lambda ^{-2}q^{1,0}+\lambda^{2}q^{0,1},$$for
$\lambda\in\C\backslash\{0\}$. The fact that $q$ takes values in
$\wedge^{2}V\oplus \wedge^{2}V^{\perp}$ establishes, in particular,
$$\mathcal{D}^{d^{\lambda,q}_{V}}_{V}=\mathcal{D}_{V}+(\lambda^{-2}-1)q^{1,0}+(\lambda^{0,1}-1)q^{0,1},$$
whereas
$$\mathcal{N}^{d^{\lambda,q}_{V}}_{V}=\lambda^{-1}\mathcal{N}_{V}^{1,0}+\lambda\mathcal{N}_{V}^{0,1},$$
and, therefore, $$(d^{\lambda
,q}_{V})^{\mu,q_{\lambda}}_{V}=d^{\lambda\mu ,q}_{V},$$for all
$\lambda,\mu\in \C\backslash \{0\}$. From the flatness of
$d^{\lambda,q}_{V}$, for all $\lambda\in\C\backslash \{0\}$, we
conclude that of $(d^{\lambda ,q}_{V})^{\mu,q_{\lambda}}_{V}$, for
all $\lambda,\mu\in \C\backslash \{0\}$ and, therefore, that $V$ is
$d^{\lambda,q}_{V}$-perturbed harmonic, for all
$\lambda\in\backslash\C\{0\}$. We define then a spectral deformation
of $V$ into new perturbed harmonic bundles by setting, for each
$\lambda$ in $\C\{0\}$,
$$V^{\lambda}_{q}:=\phi^{\lambda,q}_{V}V.$$

\begin{thm}\cite{BQ}\label{CHspecdeform}
$V^{\lambda}_{q}$ is
$\mathrm{Ad}_{\phi^{\lambda}_{q}}q_{\lambda}$-perturbed harmonic,
for each $\lambda\in \C\{0\}$.
\end{thm}

A deformation on the level of constrained Willmore surfaces follows:

\begin{thm}\cite{BQ}
For each $\lambda\in\C\backslash\{0\}$,
$(\phi^{\lambda,q}_{V}\,\Lambda^{1,0},\phi^{\lambda,q}_{V}\,\Lambda^{0,1})$
is a $\mathrm{Ad}_{\phi^{\lambda}_{q}}(q_{\lambda})$-constrained
Willmore surface, admitting $V^{\lambda}_{q}$ as a central sphere
congruence. Furthermore, if $(\Lambda^{1,0},\Lambda^{0,1})$ is a
real constrained Willmore surface, then so is
$$\Lambda^{\lambda}_{q}:=\phi^{\lambda,q}_{V}\,\Lambda,$$
for all $\lambda\in S^{1}$.
\end{thm}
\begin{proof}
By \eqref{eq:multiplierspecifics}, together with the isotropy of
$\Lambda^{i,j}$, for $i\neq j\in\{0,1\}$, we have $q\Lambda=0$. On
the other hand, the centrality of $V$ with respect to
$(\Lambda^{1,0},\Lambda^{0,1})$ gives
$\mathcal{N}_{V}\Lambda=\pi_{V^{\perp}}\circ d\Lambda=0$. Hence
\begin{equation}\label{eq:dlambdacoincideswithdonLambda}
(d^{\lambda,q}_{V})_{\vert{\Gamma\Lambda}}=d_{\vert{\Gamma\Lambda}}
\end{equation}
and we conclude that
$(\phi^{\lambda,q}_{V}\,\Lambda^{1,0},\phi^{\lambda,q}_{V}\,\Lambda^{0,1})$
is still a surface. Suppose, furthermore, that
$(\Lambda^{1,0},\Lambda^{0,1})$ is a real $q$-constrained Willmore
surface. Given $\lambda\in S^{1}$, $d^{\lambda}_{q}$ is real, so
that we can choose $\phi^{\lambda,q}_{V}$ to be real, in which case
$$\overline{\phi^{\lambda,q}_{V}\Lambda}=\phi^{\lambda,q}_{V}\Lambda,$$
$\Lambda^{\lambda}_{q}$ is a real surface. It is obvious, on the
other hand, that as $V$ is a central sphere congruence of
$(\Lambda^{1,0},\Lambda^{0,1})$, $\phi^{\lambda,q}_{V}\,V$ is a
central sphere congruence of
$(\phi^{\lambda,q}_{V}\,\Lambda^{1,0},\phi^{\lambda,q}_{V}\,\Lambda^{0,1})$.
Theorem \ref{CHspecdeform} completes the proof.
\end{proof}

Note that spectral deformation corresponding to the zero multiplier
preserves the class of Willmore surfaces. 

This spectral deformation of real constrained Willmore surfaces coincides, up to
reparametrization, with the one presented in \cite{SD}, in terms of
the \textit{Schwarzian derivative} and the \textit{Hopf
differential} (see \cite[Section 6.4.1]{thesis}).

An alternative perspective on this spectral deformation of
perturbed harmonic bundles and constrained Willmore surfaces is that of a
change of flat connection on $\underline{\C}^{n+2}$: if
$V\subset(\underline{\C}^{n+2},d)$ is perturbed harmonic, then so
is $V\subset(\underline{\C}^{n+2},d^{\lambda,q}_{V})$, as well as,
if $(\Lambda^{1,0}, \Lambda^{0,1})$ is constrained Willmore [with
respect to $d$] then $(\Lambda^{1,0}, \Lambda^{0,1})$ is still constrained Willmore with respect
to $d^{\lambda,q}_{V}$, for all $\lambda\in\C\backslash\{0\}$. In
the real case, this is the interpretation of loop group theory in
\cite{burstall+calderbank}.

\subsection{Dressing action}\label{dressingaction}
We use a version of the dressing action theory of Terng and
Uhlenbeck \cite{terng+uhlenbeck} to build transformations of $V$ into new
perturbed harmonic bundles and thereafter transformations of
$(\Lambda^{1,0},\Lambda^{0,1})$ into new constrained Willmore
surfaces. For that, we give conditions on a dressing
$r(\lambda)\in\Gamma(O(\underline{\mathbb C}^{n+2}))$ such that the
gauging $r(\lambda)\circ d^{\lambda,q}_{V}\circ r(\lambda)^{-1}$ of
$d^{\lambda,q}_{V}$ by $r(\lambda)$ establishes the perturbed 
harmonicity of some bundle $\hat{V}$ from the perturbed 
harmonicity of $V$.\newline

The $\mathcal{D}_{V}$-parallelness of $V$ and $V^{\perp}$, together
with the fact that $\mathcal{N}_{V}$ intertwines $V$ and
$V^{\perp}$, whereas $q$ preserves them, makes clear that
\begin{equation}\label{eq:rholam}
d^{-\lambda,q}_{V}=\rho\circ d_{V}^{\lambda,q}\circ\rho^{-1},
\end{equation}
for $\lambda\in\C\backslash \{0\}$. Suppose we have
$r(\lambda)\in\Gamma(O(\underline{\C}^{n+2}))$ such that
$\lambda\mapsto r(\lambda)$ is rational in $\lambda$, $r$ is
holomorphic and invertible at $\lambda=0$ and $\lambda=\infty$ and
twisted in the sense that
\begin{equation}\label{eq:RcommuteRho}
\rho\,r(\lambda)\,\rho^{-1}=r(-\lambda),
\end{equation}
for $\lambda\in\mathrm{dom}(r)$. In particular, it follows that both
$r(0)$ and $r(\infty)$ commute with $\rho$, and, therefore, that
\begin{equation}\label{eq:r0rinfpreserveVperp}
r(0)_{\vert_{V}},r(\infty)_{\vert_{V}}\in\Gamma(O(V)).
\end{equation}
Define $\hat{q}\in\Omega^{1}(\wedge^{2}V\oplus \wedge^{2}V^{\perp})$
by setting
$$\hat{q}^{1,0}:=\mathrm{Ad}_{r(0)}q^{1,0},\,\,\,\,\,\,\hat{q}\,^{0,1}:=\mathrm{Ad}_{r(\infty)}q^{0,1}.$$
Define a new family of metric connections on $\underline{\C}^{n+2}$
by setting
$$\hat{d}^{\lambda,\hat{q}}_{V}:=r(\lambda)\circ
d^{\lambda,q}_{V}\circ r(\lambda)^{-1}.$$ Suppose that there exists
a holomorphic extension of
$\lambda\mapsto\hat{d}^{\lambda,\hat{q}}_{V}$ to
$\lambda\in\C\backslash\{0\}$ through metric connections on
$\underline{\C}^{n+2}$. We shall see later how to construct such
$r=r(\lambda)$, but assume, for the moment, that we have got one. In
that case, as we, crucially, verify next, the notation
$\hat{d}^{\lambda,\hat{q}}_{V}$ is not merely formal:
\begin{prop}\cite{BQ}\label{cruxh}
$$\hat{d}^{\lambda,\hat{q}}_{V}=\mathcal{D}_{V}^{\hat{d}}+\lambda^{-1}(\mathcal{N}_{V}^{\hat{d}})^{1,0}+\lambda(\mathcal{N}_{V}^{\hat{d}})^{0,1}+(\lambda^{-2}-1)\hat{q}^{1,0}+(\lambda^{2}-1)\hat{q}^{0,1},$$
for the flat metric connection
$\hat{d}:=\hat{d}^{1,\hat{q}}_{V}=\mathrm{lim}_{\lambda\rightarrow
1}r(\lambda)\circ d^{\lambda,q}_{V}\circ r(\lambda)^{-1}$ and
$\lambda\in\C\backslash \{0\}$.
\end{prop}
\begin{proof}
The fact that $r$ is holomorphic and invertible at $\lambda=0$ and that 
$(d^{\lambda,q}_{V})^{0,1}=\mathcal{D}_{V}^{0,1}+\lambda
\mathcal{N}_{V}^{0,1}+(\lambda^{2}-1)q^{0,1}$ is holomorphic on $\C$
establishes that the connection
$$(\hat{d}^{\lambda,\hat{q}}_{V})^{0,1}=r(\lambda)\circ(d^{\lambda,q}_{V})^{0,1}\circ
r(\lambda)^{-1},$$ which admits a holomorphic extension to
$\lambda\in\C\backslash\{0\}$, admits, furthermore, a holomorphic
extension to $\lambda\in\C$. Thus, locally,
$$(\hat{d}^{\lambda,\hat{q}}_{V})^{0,1}=A_{0}^{0,1}+\sum_{i\geq
1}\lambda^{i}A_{i}^{0,1},$$ with $A_{0}$ connection and
$A_{i}\in\Omega^{1}(o(\underline{\C}^{n+2}))$, for all $i$.
Considering then limits of
$$\lambda^{-2}A_{0}^{0,1}+\sum_{i\geq
1}\lambda^{i-2}A_{i}^{0,1}=r(\lambda)\circ
(\lambda^{-2}\mathcal{D}_{V}^{0,1}+\lambda^{-1}\mathcal{N}_{V}^{0,1}+(1-\lambda^{-2})q^{0,1})\circ
r(\lambda)^{-1},$$when $\lambda$ goes to infinity, we get
$$A_{2}^{0,1}+\mathrm{lim}_{\lambda\rightarrow \infty}\sum_{i\geq
3}\lambda^{i-2}A_{i}^{0,1}=\mathrm{Ad}_{r(\infty)}\,q^{0,1},$$ which
shows that $A_{i}^{0,1}=0$, for all $i\geq 3$, and that
$A_{2}^{0,1}=\hat{q}^{0,1}$. Considering now limits of
$$A_{0}^{0,1}+\lambda A_{1}^{0,1}+\lambda^{2}\hat{q}^{0,1}=r(\lambda)\circ
(\mathcal{D}_{V}^{0,1}+\lambda\mathcal{N}_{V}^{0,1}+(\lambda^{2}-1)q^{0,1})\circ
r(\lambda)^{-1},$$when $\lambda$ goes to $0$, we conclude that
$$A_{0}^{0,1}=r(0)\circ(\mathcal{D}_{V}^{0,1}-q^{0,1})\circ
r(0)^{-1}$$ and, therefore, that
$$(\hat{d}^{\lambda,\hat{q}}_{V})^{0,1}=r(0)\circ(\mathcal{D}_{V}^{0,1}-q^{0,1})\circ
r(0)^{-1} +\lambda A_{1}^{0,1}+\lambda^{2}\hat{q}^{0,1}.$$ As for
$$(\hat{d}^{\lambda,\hat{q}}_{V})^{1,0}=r(\lambda)\circ
(\mathcal{D}_{V}^{1,0}+\lambda^{-1}\mathcal{N}^{1,0}+(\lambda^{-2}-1)q^{1,0})\circ
r(\lambda)^{-1},$$ which has a pole at $\lambda=0$, we have, for
$\lambda$ away from $0$,
\begin{equation}\label{eq:dhat01holomwithpole0}
\sum_{i\geq 1}\lambda^{-i}A_{-i}^{1,0}+A^{1,0}_{0}+\sum_{i\geq
1}\lambda^{i}A_{i}^{1,0}=r(\lambda)\circ
(\mathcal{D}_{V}^{1,0}+\lambda^{-1}\mathcal{N}^{1,0}+(\lambda^{-2}-1)q^{1,0})\circ
r(\lambda)^{-1},
\end{equation}
with $A_{-i}^{1,0}\in\Omega^{1}(o(\underline{\C}^{n+2}))$, for all
$i\geq 1$. Considering limits of \eqref{eq:dhat01holomwithpole0}
when $\lambda$ goes to infinity, shows that $A_{i}^{1,0}=0$, for all
$i\geq 1$, and that
$$A^{1,0}_{0}=r(\infty)\circ(\mathcal{D}_{V}^{1,0}-q^{1,0})\circ
r(\infty)^{-1}.$$ Multiplying then both members of equation
\eqref{eq:dhat01holomwithpole0} by $\lambda^{2}$ and considering
limits when $\lambda$ goes to $0$, we conclude that
$A^{1,0}_{-2}=\hat{q}^{1,0}$ and that $A_{-i}^{1,0}=0$, for all
$i\geq 3$, and, ultimately, that
$$(\hat{d}^{\lambda,\hat{q}}_{V})^{1,0}=r(\infty)\circ(\mathcal{D}^{1,0}_{V}-q^{1,0})\circ
r(\infty)^{-1}+\lambda^{-1}A_{-1}^{1,0}+\lambda^{-2}\hat{q}^{1,0}.$$
Thus
\begin{eqnarray*}
\hat{d}_{V}^{\lambda,\hat{q}}&=&r(0)\circ(\mathcal{D}_{V}^{0,1}-q^{0,1}+q^{1,0})\circ r(0)^{-1}\\&& \mbox{}+r(\infty)\circ(\mathcal{D}^{1,0}_{V}-q^{1,0}+q^{0,1})\circ r(\infty)^{-1}\\
&& \mbox{}+ \lambda^{-1}A_{-1}^{1,0}+\lambda
A_{1}^{0,1}+(\lambda^{-2}-1)\hat{q}^{1,0}+(\lambda^{2}-1)\hat{q}^{0,1},
\end{eqnarray*}
for $\lambda\in\C\backslash \{0\}$, and, in particular,
$$\hat{d}=r(0)\circ(\mathcal{D}_{V}^{0,1}-q^{0,1}+q^{1,0})\circ r(0)^{-1}+r(\infty)\circ(\mathcal{D}^{1,0}_{V}-q^{1,0}+q^{0,1})\circ r(\infty)^{-1}+
A_{-1}^{1,0}+A_{1}^{0,1}.$$ The fact that $r(0)$ and $r(\infty)$
(and so $r(0)^{-1}$ and $r(\infty)^{-1}$), as well as $q$, preserve
$V$ and $V^{\perp}$, together with the
$\mathcal{D}_{V}$-parallelness of $V$ and of $V^{\perp}$, shows that
$\hat{d}-(A_{-1}^{1,0}+A_{1}^{0,1})$ preserves $\Gamma(V)$ and
$\Gamma(V^{\perp})$. On the other hand, equations \eqref{eq:rholam}
and \eqref{eq:RcommuteRho} combine to give
$$\hat{d}^{-\lambda,\hat{q}}_{V}=\rho\circ\hat{d}^{\lambda,\hat{q}}_{V}\circ\rho^{-1},$$
for all $\lambda\in\C\backslash\{0\}$ away from the poles of $r$ and
then, by continuity, on all of $\C\backslash\{0\}$. The particular
case of $\lambda=1$ gives $\rho
(A_{-1}^{1,0}+A_{1}^{0,1})_{\vert_{V}}=-(A_{-1}^{1,0}+A_{1}^{0,1})_{\vert_{V}}$
and $\rho
(A_{-1}^{1,0}+A_{1}^{0,1})_{\vert_{V^{\perp}}}=-(A_{-1}^{1,0}+A_{1}^{0,1})_{\vert_{V^{\perp}}}$,
showing that $$A_{-1}^{1,0}+A_{1}^{0,1}\in\Omega^{1}(V\wedge
V^{\perp}).$$ We conclude that
\begin{equation}\label{eq:mathcalDdosr}
r(0)\circ(\mathcal{D}_{V}^{0,1}-q^{0,1}+q^{1,0})\circ
r(0)^{-1}+r(\infty)\circ(\mathcal{D}^{1,0}_{V}-q^{1,0}+q^{0,1})\circ
r(\infty)^{-1}=\mathcal{D}_{V}^{\hat{d}}
\end{equation}
and $$A_{-1}^{1,0}=(\mathcal{N}_{V}^{\hat{d}})^{1,0},\,\,\,
A_{1}^{0,1}=(\mathcal{N}_{V}^{\hat{d}})^{0,1},$$ completing the
proof.
\end{proof}

The flatness of $d^{\lambda,q}_{V}$ for all $\lambda\in\C\backslash
\{0\}$ establishes that of $\hat{d}^{\lambda,\hat{q}}_{V}$, for all
non-zero $\lambda$ away from the poles of $r$ and then, by
continuity, for all $\lambda\in\C\backslash\{0\}$. By Proposition
\ref{cruxh}, we conclude that $V$ is $(\hat{q},\hat{d})$-perturbed 
harmonic. Suppose $1\in\mathrm{dom}(r)$. By Lemma
\ref{changeconnection}, it follows that:
\begin{thm}\cite{BQ}\label{CHtransf}
$r(1)^{-1}V$ is a $\mathrm{Ad}_{r(1)^{-1}}\,\hat{q}$-perturbed 
harmonic bundle.
\end{thm}

Note that this transformation preserves the harmonicity condition.

A transformation on the level of constrained Willmore surfaces
follows, with some extra condition, as we shall see next. Set
$$\hat{\Lambda}^{1,0}:=r(\infty)\Lambda^{1,0},\,\,\,\,\,\,\hat{\Lambda}^{0,1}:=r(0)\Lambda^{0,1}$$
and $$\hat{\Lambda}=\hat{\Lambda}^{1,0}\cap \hat{\Lambda}^{0,1}.$$
Suppose, furthermore, that
\begin{equation}\label{eq:condondet}
\mathrm{det}\,r(0)_{\vert_{V}}=\mathrm{det}\,r(\infty)_{\vert_{V}}.
\end{equation}
Then:

\begin{thm}\cite{BQ}\label{eq:thm8.4.2paraja}
$(r(1)^{-1}\hat{\Lambda}^{1,0},r(1)^{-1}\hat{\Lambda}^{0,1})$ is a
$\mathrm{Ad}_{r(1)^{-1}}\hat{q}$-constrained Willmore surface
admitting $r(1)^{-1}V$ as a central sphere congruence.
\end{thm}

\begin{proof}
First of all, note that, by \eqref{eq:multiplierspecifics},
$\hat{q}^{i,j}\in\Omega^{i,j}(\wedge^{2}\hat{\Lambda}^{j,i})$ and,
therefore,
$$(\mathrm{Ad}_{r(1)^{-1}}\hat{q})^{i,j}\in\Omega^{i,j}(\wedge^{2}r(1)^{-1}\hat{\Lambda}^{j,i}),$$
for $i\neq j\in\{0,1\}$. In the light of Theorem \ref{CHtransf}, we
are left to verify that
$(r(1)^{-1}\hat{\Lambda}^{1,0},r(1)^{-1}\hat{\Lambda}^{0,1})$ is a
surface admitting $r(1)^{-1}V$ as a central sphere congruence.

The fact that $\Lambda^{1,0}$ and $\Lambda^{0,1}$ are rank $2$
isotropic subbundles of $V$ ensures that so are
$\hat{\Lambda^{1,0}}$ and $\hat{\Lambda^{0,1}}$, as $r(0)$ and
$r(\infty)$ are orthogonal transformations and preserve $V$. To see
that $\hat{\Lambda}$ is $\mathrm{rank}\,1$, we use some well-known
facts about the Grassmannian $\mathcal{G}_{W}$ of isotropic
$2$-planes in a complex $4$-dimensional space $W$: it has two
components, each an orbit of the special orthogonal group $SO(W)$,
intertwined by the action of elements of $O(W)\backslash SO(W)$, and
for which any element intersects any element of the other component
in a line while distinct elements of the same component have trivial
intersection. Since $\mathrm{rank}\,\Lambda=1$, $\Lambda^{1,0}_{p}$
and $\Lambda^{0,1}_{p}$ lie in different components of
$\mathcal{G}_{V_{p}}$ and the hypothesis \eqref{eq:condondet}
ensures that the same is true of $\hat{\Lambda}^{1,0}_{p}$ and
$\hat{\Lambda}^{0,1}_{p}$, for all $p$.

We are left to verify that
\begin{equation}\label{eq:chapeusds}
\hat{d}^{1,0}\Gamma(\hat{\Lambda})\subset\Omega^{1}(\hat{\Lambda}^{1,0}),\,\,\,\,\,\,\hat{d}^{0,1}\Gamma(\hat{\Lambda})\subset\Omega^{1}(\hat{\Lambda}^{0,1})
\end{equation}
and that (recall equation \eqref{eq:curlyNphi})
\begin{equation}\label{eq:chapeusNs}
(\mathcal{N}_{V}^{\hat{d}})^{1,0}\hat{\Lambda}^{0,1}=0=(\mathcal{N}_{V}^{\hat{d}})^{0,1}\hat{\Lambda}^{1,0}.
\end{equation}
Equation \eqref{eq:chapeusNs} forces
$\mathcal{N}_{V}^{\hat{d}}\hat{\Lambda}=0$, in which situation,
\eqref{eq:chapeusds} reads
$$(\mathcal{D}_{V}^{\hat{d}})^{1,0}\Gamma(\hat{\Lambda})\subset\Omega^{1}(\hat{\Lambda}^{1,0}),\,\,\,\,\,\,(\mathcal{D}_{V}^{\hat{d}})^{0,1}\Gamma(\hat{\Lambda})\subset\Omega^{1}(\hat{\Lambda}^{0,1}),$$
which, in its turn, follows from
\begin{equation}\label{eq:vcsl'1'08763trfgvhbjvuf364rgbcn}
(\mathcal{D}_{V}^{\hat{d}})^{1,0}\Gamma(\hat{\Lambda}^{1,0})\subset\Omega^{1}(\hat{\Lambda}^{1,0}),\,\,\,\,\,\,(\mathcal{D}_{V}^{\hat{d}})^{0,1}\Gamma(\hat{\Lambda}^{0,1})\subset\Omega^{1}(\hat{\Lambda}^{0,1}).
\end{equation}
It is \eqref{eq:chapeusNs} and
\eqref{eq:vcsl'1'08763trfgvhbjvuf364rgbcn} that we shall establish.

First of all, note that, according to \eqref{eq:mathcalDdosr},
$$(\mathcal{D}_{V}^{\hat{d}})^{1,0}=r(\infty)\circ(\mathcal{D}_{V}^{1,0}-q^{1,0})\circ
r(\infty)^{-1}+\hat{q}^{1,0}$$ and
$$(\mathcal{D}^{\hat{d}}_{V})^{0,1}=r(0)\circ(\mathcal{D}_{V}^{0,1}-q^{0,1})\circ r(0)^{-1}+\hat{q}^{0,1}.$$
Now $q^{1,0}$ takes values in $\Lambda\wedge\Lambda^{0,1}$, so
$q^{1,0}\Lambda^{1,0}\subset\Lambda\subset\Lambda^{1,0}$, by the
isotropy of $\Lambda^{1,0}$. On the other hand, since
$\mathrm{rank}\,\hat{\Lambda}=1$, we have
$\wedge^{2}\hat{\Lambda}^{0,1}=\hat{\Lambda}\wedge\hat{\Lambda}^{0,1}$
and, therefore,
$\hat{q}^{1,0}\hat{\Lambda}^{1,0}\subset\hat{\Lambda}\subset\hat{\Lambda}^{1,0}$.
Together with Lemma \ref{curlyDfacts}, this establishes the
$(1,0)$-part of \eqref{eq:vcsl'1'08763trfgvhbjvuf364rgbcn}. A
similar argument establishes the $(0,1)$-part of it.

Finally, we establish \eqref{eq:chapeusNs}. According to Proposition
\ref{cruxh},
\begin{eqnarray*}
(\mathcal{N}^{\hat{d}}_{V})^{1,0}&=&\mathrm{lim}_{\lambda\rightarrow
0}\,\lambda((\hat{d}_{V}^{\lambda,\hat{q}})^{1,0}-(\mathcal{D}^{\hat{d}}_{V})^{1,0}-(\lambda^{-2}-1)\hat{q}^{1,0})\\&=&\mathrm{lim}_{\lambda\rightarrow
0}\,\lambda((\hat{d}_{V}^{\lambda,\hat{q}})^{1,0}-\lambda^{-2}\hat{q}^{1,0})\\&=&\mathrm{lim}_{\lambda\rightarrow
0}\,(r(\lambda)\circ(\lambda\, (d^{\lambda,q}_{V})^{1,0})\circ
r(\lambda)^{-1}-\lambda^{-1}\mathrm{Ad}_{r(0)}q^{1,0})\\&=&\mathrm{Ad}_{r(0)}\mathcal{N}^{1,0}_{V}+\mathrm{lim}_{\lambda\rightarrow
0}\,\frac{1}{\lambda}(\mathrm{Ad}_{r(\lambda)}-\mathrm{Ad}_{r(0)})q^{1,0}.
\end{eqnarray*}
so that
\begin{equation}\label{eq:calNhatd0,1}
(\mathcal{N}^{\hat{d}}_{V})^{1,0}=\mathrm{Ad}_{r(0)}\mathcal{N}^{1,0}_{V}+\frac{d}{d\lambda}_{\vert_{\lambda=0}}\mathrm{Ad}_{r(\lambda)}q^{1,0};
\end{equation}
and, similarly,
\begin{eqnarray*}
(\mathcal{N}^{\hat{d}}_{V})^{0,1}&=&\mathrm{lim}_{\lambda\rightarrow
\infty}\,\lambda^{-1}((\hat{d}_{V}^{\lambda,\hat{q}})^{0,1}-(\mathcal{D}^{\hat{d}}_{V})^{0,1}-(\lambda^{2}-1)\hat{q}^{0,1})\\&=&
\mathrm{Ad}_{r(\infty)}\mathcal{N}^{0,1}_{V}+\mathrm{lim}_{\lambda\rightarrow
\infty}\,(r(\lambda)\circ \lambda q^{0,1}\circ
r(\lambda)^{-1}-\lambda\mathrm{Ad}_{r(\infty)}q^{0,1})
\end{eqnarray*}
and, therefore,
\begin{equation}\label{eq:calNhatd1,0}
(\mathcal{N}^{\hat{d}}_{V})^{0,1}=\mathrm{Ad}_{r(\infty)}\mathcal{N}^{0,1}_{V}+\frac{d}{d\lambda}_{\vert_{\lambda=0}}\mathrm{Ad}_{r(\lambda^{-1})}q^{0,1}.
\end{equation}
Furthermore, by \eqref{eq:calNhatd0,1},
$$(\mathcal{N}^{\hat{d}}_{V})^{1,0}= \mathrm{Ad}_{r(0)}(\mathcal{N}^{1,0}+[r(0)^{-1}\frac{d}{d\lambda}_{\vert_{\lambda=0}}r(\lambda),q^{1,0}]).$$
The centrality of $V$ with respect to
$(\Lambda^{1,0},\Lambda^{0,1})$ establishes, in particular,
$\mathcal{N}_{V}^{1,0}\Lambda^{0,1}=0$, whilst the isotropy of
$\Lambda^{0,1}$ ensures, in particular, that
$q^{1,0}\Lambda^{0,1}=0$. Hence
$$\mathrm{Ad}_{r(0)}(\mathcal{N}^{1,0}+r(0)^{-1}\frac{d}{d\lambda}_{\vert_{\lambda=0}}r(\lambda)\,q^{1,0})\hat{\Lambda}^{0,1}=0.$$
On the other hand, differentiation of  $r(\lambda)^{-1}=\rho\,
r(-\lambda)^{-1} \rho$, derived from equation
\eqref{eq:RcommuteRho}, gives
$$-r(\lambda)^{-1}\frac{d}{dk}_{\vert_{k=\lambda}}r(k)\,r(\lambda)^{-1}=\rho\, r(-\lambda)^{-1}\frac{d}{dk}_{\vert_{k=-\lambda}}r(k)\,r(-\lambda)^{-1}\rho,$$
or, equivalently,
$$\rho\,r(\lambda)^{-1}\frac{d}{dk}_{\vert_{k=\lambda}}r(k)\rho=-r(-\lambda)^{-1}\frac{d}{dk}_{\vert_{k=-\lambda}}r(k)\,r(-\lambda)^{-1}\rho\,r(\lambda)\rho,$$
and, therefore, yet again by equation \eqref{eq:RcommuteRho},
$$\rho\,r(\lambda)^{-1}\frac{d}{dk}_{\vert_{k=\lambda}}r(k)\rho=-r(-\lambda)^{-1}\frac{d}{dk}_{\vert_{k=-\lambda}}r(k).$$
Evaluation at $\lambda=0$ shows then that $$\rho\,
r(0)^{-1}\frac{d}{d\lambda}_{\vert_{\lambda=0}}r(\lambda)\rho=-r(0)^{-1}\frac{d}{d\lambda}_{\vert_{\lambda=0}}r(\lambda).$$
Equivalently,
\begin{equation}\label{eq:rrsat0}
r(0)^{-1}\frac{d}{d\lambda}_{\vert_{\lambda=0}}r(\lambda)\in\Gamma(V\wedge
V^{\perp}).
\end{equation}
Since $qV^{\perp}=0$, we conclude that
$$Ad_{r(0)}(q^{1,0}r(0)^{-1}\frac{d}{d\lambda}_{\vert_{\lambda=0}}r(\lambda))\hat{\Lambda}^{0,1}=0$$
and, ultimately, that
$(\mathcal{N}_{V}^{\hat{d}})^{1,0}\hat{\Lambda}^{0,1}=0$. A similar
argument near $\lambda=\infty$ establishes
$(\mathcal{N}_{V}^{\hat{d}})^{0,1}\hat{\Lambda}^{1,0}=0$, completing
the proof.
\end{proof}

\subsection{B\"{a}cklund transformation}
We now construct $r=r(\lambda)$ satisfying the hypothesis of the
previous section. As the philosophy underlying the work of C.-L.
Terng and K. Uhlenbeck \cite{uhlenbeck} suggests, we consider linear
fractional transformations. As we shall see, a two-step process will
produce a desired $r$.\newline

Given $\alpha\in\C\backslash \{-1,0,1\}$ and $L$ a null line
subbundle of $\underline{\C}^{n+2}$ such that, locally, $\rho L\cap
L^{\perp}=\{0\}$, set
$$p_{\alpha,L}(\lambda):=
I\left\{
\begin{array}{ll} \frac{\alpha-\lambda}{\alpha+\lambda} & \mbox{$\mathrm{on}\,L$}\\ 1 &
\mbox{$\mathrm{on}\,(L\oplus \rho
L)^{\perp}$}\\\frac{\alpha+\lambda}{\alpha-\lambda}&
\mbox{$\mathrm{on}\,\rho L$}\end{array}\right.$$ and
$$q_{\alpha,L}(\lambda):= I\left\{
\begin{array}{ll}\frac{\lambda-\alpha}{\lambda+\alpha} & \mbox{$\mathrm{on}\,L$}\\ 1 &
\mbox{$\mathrm{on}\,(L\oplus \rho
L)^{\perp}$}\\\frac{\lambda+\alpha}{\lambda-\alpha}&
\mbox{$\mathrm{on}\,\rho L$}\end{array}\right.,$$for
$\lambda\in\C\backslash\{\pm\alpha\}$, defining in this way maps
$$p_{\alpha,L}, q_{\alpha,L}:\C\backslash\{\pm\alpha\}\rightarrow
\Gamma(O(\underline{\C}^{n+2}))$$ that, clearly, extend
holomorphically to $\mathbb{P}^{1}\backslash\{\pm \alpha\}$, by
setting
$$p_{\alpha,L}(\infty):=
I\left\{
\begin{array}{ll} -1 & \mbox{$\mathrm{on}\,L$}\\ 1 &
\mbox{$\mathrm{on}\,(L\oplus \rho L)^{\perp}$}\\-1&
\mbox{$\mathrm{on}\,\rho L$}\end{array}\right.$$ and
$$q_{\alpha,L}(\infty):=I.$$ Obviously, $p_{\alpha,L}(\infty)$ and
$q_{\alpha,L}(\infty)$ do not depend on $\alpha$. For further
reference, note that, for all
$\lambda\in\C\backslash\{\pm\alpha,0\}$, we have
\begin{equation}\label{eq:conjap}
p_{\alpha,L}(\lambda)=q_{\alpha^{-1},L}(\lambda^{-1}),
\end{equation}
whilst $$p_{\alpha,L}(0)=q_{\alpha,L}(\infty),\,\,\,\,p_{\alpha,L}(\infty)=q_{\alpha,L}(0).$$

The isometry $\rho=\rho ^{-1}$ intertwines $L$ and $\rho L$ and,
therefore, preserves $(L\oplus \rho L)^{\perp}$, which makes clear
that $\rho\circ p_{\alpha,L}(\lambda)$ and
$p_{\alpha,L}(\lambda)^{-1}\circ\rho$ coincide in $L$, $\rho L$ and
$(L\oplus \rho L)^{\perp}$ and, therefore,
\begin{equation}\label{twistespandq}
\rho\,p_{\alpha,L}(\lambda)\rho=p_{\alpha,L}(\lambda)^{-1}=p_{\alpha,L}(-\lambda),
\end{equation}
and, similarly,
\begin{equation}\label{twistespandq2}
\rho\,q_{\alpha,L}(\lambda)\rho=q_{\alpha,L}(\lambda)^{-1}=q_{\alpha,L}(-\lambda),
\end{equation}
for all $\lambda$ - both $p_{\alpha,L}$ and $q_{\alpha,L}$ are
twisted in the sense of Section \ref{dressingaction}.

Since $\rho L$ is not orthogonal to $L$, $\rho L\neq L$, so $L$ is
not a subbundle of $V$ and, therefore,
$\mathrm{rank}\,V\cap(L\oplus\rho L)=1$. We conclude that
$$p_{\alpha,L}(\infty)_{\vert_{V}}=q_{\alpha,L}(0)_{\vert_{V}}
=I\left\{
\begin{array}{ll} -1 & \mbox{$\mathrm{on}\,V\cap(L\oplus\rho L)$}\\ 1 &
\mbox{$\mathrm{on}\,V\cap(L\oplus\rho
L)^{\perp}$}\end{array}\right.$$ has determinant $-1\neq
1=\mathrm{det}\,p_{\alpha,L}(0)_{\vert_{V}}=\mathrm{det}\,q_{\alpha,L}(\infty)_{\vert_{V}}$,
so we cannot take $p_{\alpha,L}=r$ or $q_{\alpha,L}=r$ in the
analysis of Section \ref{dressingaction}. However, we will be able
to take $r=q_{\beta,\hat{L}}p_{\alpha,L}$, for suitable $\beta$ and
$\hat{L}$, as we shall see.

Now choose $\alpha\in\C\backslash \{-1,0,1\}$ and $L^{\alpha}$ a
$d^{\alpha,q}_{V}$-parallel null line subbundle of
$\underline{\C}^{n+2}$ such that, locally,
\begin{equation}\label{eq:LrhoL}
\rho L^{\alpha}\cap (L^{\alpha})^{\perp}=\{0\}.
\end{equation}
Such a bundle $L^{\alpha}$ can be obtained by
$d^{\alpha,q}_{V}$-parallel transport of
$l^{\alpha}_{p}\in\C^{n+2}$, with $l^{\alpha}_{p}$ null
 and non-orthogonal to $\rho_{p}l^{\alpha}_{p}$, for some $p\in\Sigma$.

\begin{lemma}\cite{BQ}\label{lemmaext}
There exists a holomorphic extension of $$\lambda\mapsto
d^{\lambda,q}_{p_{_{\alpha,L^{\alpha}}}}:=p_{\alpha,L^{\alpha}}(\lambda)\circ
d^{\lambda,q}_{V}\circ p_{\alpha,L^{\alpha}}(\lambda)^{-1}$$ to
$\lambda\in\C\backslash\{0\}$ through metric connections on
$\underline{\C}^{n+2}$.
\end{lemma}
\begin{proof}
We prove holomorphicity at $\lambda=\alpha$. For
$\lambda\in\C\backslash \{0,\alpha\}$, write
$$d^{\lambda,q}_{V}=d^{\alpha,q}_{V}+(\lambda-\alpha)A(\lambda),$$ with $\lambda\mapsto
A(\lambda)\in\Omega^{1}(o(\underline{\C}^{n+2}))$ holomorphic.
Decompose $d^{\alpha,q}_{V}=D+\beta$ according to the decomposition
$$\underline{\C}^{n+2}=(L^{\alpha}\oplus\rho L^{\alpha})\oplus
(L^{\alpha}\oplus\rho L^{\alpha})^{\perp}.$$ The fact that
$d^{\alpha,q}_{V}$ is a metric connection establishes
$$d^{\alpha,q}_{V}\,\Gamma (\rho L^{\alpha})\subset \Omega^{1}(\rho
L^{\alpha})^{\perp},$$ as well as
$$d^{\alpha,q}_{V}\,\Gamma ((L^{\alpha}\oplus\rho
L^{\alpha})^{\perp})\subset \Omega^{1} (L^{\alpha})^{\perp},$$ 
in view of the
$d^{\alpha,q}_{V}$-parallelness of $L^{\alpha}$. By
\eqref{eq:LrhoL}, we conclude that the $1$-form
$\beta\in\Omega^{1}((L^{\alpha}\oplus\rho L^{\alpha})\wedge
(L^{\alpha}\oplus\rho L^{\alpha})^{\perp})$ takes values in
$L^{\alpha}\wedge (L^{\alpha}\oplus\rho L^{\alpha})^{\perp}$. Hence
$$p_{\alpha,L^{\alpha}}(\lambda)\circ \beta\circ
p_{\alpha,L^{\alpha}}(\lambda)^{-1}=\frac{\alpha-\lambda}{\alpha+\lambda}\,\beta.$$
On the other hand, $$p_{\alpha,L^{\alpha}}(\lambda)\circ D\circ
p_{\alpha,L^{\alpha}}(\lambda)^{-1}=D,$$ as $L^{\alpha}$, $\rho
L^{\alpha}$ and $(L^{\alpha}\oplus\rho L^{\alpha})^{\perp}$ are all
$D$-parallel. Thus
$$d^{\lambda,q}_{p_{_{\alpha,L}}}=D+\frac{\alpha-\lambda}{\alpha+\lambda}\,\beta+(\lambda-\alpha)\,p_{\alpha,L^{\alpha}}(\lambda)\,A(\lambda)\,p_{\alpha,L^{\alpha}}(\lambda)^{-1}.$$
Lastly, note that, by the skew-symmetry of $A(\lambda)$, we have
$A(\lambda)L^{\alpha}\subset (L^{\alpha})^{\perp}$ and
$A(\lambda)\rho L^{\alpha}\subset (\rho L^{\alpha})^{\perp}$ and,
therefore, by \eqref{eq:LrhoL}, $$A(\lambda)L^{\alpha}\subset
L^{\alpha}\oplus (L^{\alpha}\oplus\rho L^{\alpha})^{\perp}$$ and
$$A(\lambda)\rho L^{\alpha}\subset \rho L^{\alpha}\oplus
(L^{\alpha}\oplus\rho L^{\alpha})^{\perp}.$$ We conclude that
$(\lambda-\alpha)\,\mathrm{Ad}_{p_{\alpha,L^{\alpha}}(\lambda)}\,A(\lambda)$
has at most a simple pole at $\lambda=-\alpha$ and, therefore, that
$d^{\lambda,q}_{p_{_{\alpha,L}}}$ is holomorphic at
$\lambda=\alpha$. Furthermore, the fact that $D$ is a metric
connection establishes that so is $d^{\lambda,q}_{p_{_{\alpha,L}}}$,
in view of the skew-symmetry of $A(\lambda)$ and of $\beta$.

Holomorphicity at $\lambda=-\alpha$ can either be proved in the same
way, having in consideration that the
$d^{\alpha,q}_{V}$-parallelness of $L^{\alpha}$ establishes the
$d^{-\alpha,q}_{V}$-parallelness of $\rho L^{\alpha}$, or by
exploiting the symmetry $\lambda\mapsto -\lambda$.
\end{proof}

\begin{rem}
$1)$ The same argument establishes the existence of a holomorphic
extension of $$\lambda\mapsto
d^{\lambda,q}_{q_{_{\alpha,L^{\alpha}}}}:=q_{\alpha,L^{\alpha}}(\lambda)\circ
d^{\lambda,q}_{V}\circ q_{\alpha,L^{\alpha}}(\lambda)^{-1}$$ to
$\lambda\in\C\backslash\{0\}$ through metric connections on
$\underline{\C}^{n+2}$.\newline

$2)$ This argument uses nothing about the precise form of
$d^{\lambda,q}_{V}$, only that it is holomorphic near
$\lambda=\pm\alpha$.
\end{rem}

Now we can iterate the procedure starting with the connections
$d^{\lambda,q}_{p_{_{\alpha,L^{\alpha}}}}$. Choose
$\beta\neq\pm\alpha$ in $\C\backslash\{-1,0,1\}$ and $L^{\beta}$ a
$d^{\beta,q}_{V}$-parallel null line subbundle of
$\underline{\C}^{n+2}$. The fact that
$$p_{\alpha,L^{\alpha}}:(\underline{\C}^{n+2},d^{\lambda,q}_{V})\rightarrow
(\underline{\C}^{n+2},d^{\lambda,q}_{p_{_{\alpha,L^{\alpha}}}})$$
preserves connections establishes the
$d^{\beta,q}_{p_{_{\alpha,L^{\alpha}}}}$-parallelness of
$$\hat{L}^{\beta}_{\alpha}:=p_{\alpha,L^{\alpha}}(\beta)L^{\beta}.$$
Choose $L^{\beta}$ satisfying, furthermore, $\rho
\hat{L}^{\beta}_{\alpha}\cap(\hat{L}^{\beta}_{\alpha})^{\perp}=\{0\}$. Such
a bundle $L^{\beta}$ can be obtained by $d^{\beta,q}_{V}$-parallel
transport of $l^{\alpha}_{p}\in L^{\alpha}_{p}$, with
$l^{\alpha}_{p}$ non-zero, non-orthogonal to
$\rho_{p}l^{\alpha}_{p}$, for some
 $p\in\Sigma$. Indeed, by \eqref{twistespandq},
$$(\rho_{p} p_{\alpha,L^{\alpha}_{p}}(\beta)l^{\alpha}_{p},
p_{\alpha,L^{\alpha}_{p}}(\beta)l^{\alpha}_{p})=
(p_{\alpha,L^{\alpha}_{p}}(\beta)^{-1}\rho_{p}l^{\alpha}_{p},
p_{\alpha,L^{\alpha}_{p}}(\beta)l^{\alpha}_{p})=\frac{(\alpha-\beta)^{2}}{(\alpha+\beta)^{2}}\,(\rho_{p}l^{\alpha}_{p},
l^{\alpha}_{p}).$$
It follows that
$$\lambda\mapsto
q_{\beta,\hat{L}^{\beta}_{\alpha}}(\lambda)\,p_{\alpha,L^{\alpha}}(\lambda)\circ
d^{\lambda,q}_{V}\circ p_{\alpha,L^{\alpha}}(\lambda)^{-1}\,
q_{\beta,\hat{L}^{\beta}_{\alpha}}(\lambda)^{-1}$$ admits a
holomorphic extension to $\lambda\in\C\backslash\{0\}$ through
metric connections on $\underline{\C}^{n+2}$ and, furthermore, that
$$r^{*}:=q_{\beta,\hat{L}^{\beta}_{\alpha}}p_{\alpha,L^{\alpha}}$$
satisfies all the hypothesis of Section \ref{dressingaction} on $r$.
Set
$$q^{*}:=\mathrm{Ad}_{r^{*}(1)^{-1}}(\mathrm{Ad}_{r^{*}(0)}q^{1,0}+
\mathrm{Ad}_{r^{*}(\infty)}q^{0,1}).$$

\begin{defn}\label{BTdefn}
The $q^{*}$-perturbed harmonic bundle $$V^{*}:=r^{*}(1)^{-1}\,V$$ is
said to be the \emph{B\"{a}cklund transform of $V$ of parameters
$\alpha,\beta,L^{\alpha},L^{\beta}$}. The $q^{*}$-constrained
Willmore surface
$$(\Lambda^{1,0}, \,\Lambda^{0,1})^{*}:=(r^{*}(1)^{-1}\,r^{*}(\infty)\Lambda^{1,0},r^{*}(1)^{-1}\,r^{*}(0)\Lambda^{0,1})$$
is said to be the \emph{B\"{a}cklund transform of
$(\Lambda^{1,0},\Lambda^{0,1})$ of parameters
$\alpha,\beta,L^{\alpha},L^{\beta}$}.
\end{defn}

Note that transformations corresponding to the zero multiplier preserve the class of Willmore surfaces. 

For further reference, set
$$((\Lambda^{*})^{1,0},(\Lambda^{*})^{0,1}):=(r^{*}(1)^{-1}\,r^{*}(\infty)\Lambda^{1,0},r^{*}(1)^{-1}\,r^{*}(0)\Lambda^{0,1}).$$

\subsubsection{Bianchi permutability}
Next we establish a \textit{Bianchi permutability} of type $p$ and
type $q$ transformations, showing that starting the procedure above
with the connections $d^{\lambda,q}_{q_{_{\beta,L^{\beta}}}}$ (when
defined), instead of $d^{\lambda,q}_{p_{_{\alpha,L^{\alpha}}}}$,
produces the same transforms. The underlying argument will play a
crucial role when investigating the preservation of reality
conditions by B\"{a}cklund transformation, in the next
section.\newline

Suppose $\rho L^{\beta}\cap (L^{\beta})^{\perp}=\{0\}$ and set
$\tilde{L}^{\alpha}_{\beta}:=q_{\beta,L^{\beta}}(\alpha)L^{\alpha}$.
Suppose, furthermore, that $\rho \tilde{L}^{\alpha}_{\beta}\cap
(\tilde{L}^{\alpha}_{\beta})^{\perp}=\{0\}$ (this is
certainly the case for $L^{\beta}$ obtained by
$d^{\beta,q}_{V}$-parallel transport of $l^{\alpha}_{p}\in
L^{\alpha}_{p}$, with $l^{\alpha}_{p}$ non-zero, non-orthogonal to
$\rho_{p}l^{\alpha}_{p}$, for some
 $p\in\Sigma$). Analogously to $r^{*}$,
we verify that
$$\hat{r}^{*}:=p_{\alpha,\tilde{L}^{\alpha}_{\beta}}q_{\beta,L^{\beta}}$$
satisfies all the hypothesis of Section \ref{dressingaction} on $r$.
The next result, relating $\hat{r}^{*}$ to $r^{*}$, will be crucial
in all that follows.

\begin{lemma}\cite{BQ}\label{rstarvshatrstar}
\begin{equation}\label{eq:rel}
\hat{r}^{*}=K\,r^{*},
\end{equation}
for $K:=q_{\beta,L^{\beta}}(0)\,
q_{\beta,\hat{L}^{\beta}_{\alpha}}(0)$.
\end{lemma}

The proof of the lemma we present next will be based on the
following result:

\begin{lemma}\cite{IS}\label{TrioDeHolomorfia}
Let
$\gamma(\lambda)=\lambda\,\pi_{L_{1}}+\pi_{L_{0}}+\lambda^{-1}\,\pi_{L_{-1}}$
and
$\hat{\gamma}(\lambda)=\lambda\,\pi_{\hat{L}_{1}}+\pi_{\hat{L}_{0}}+\lambda^{-1}\,\pi_{\hat{L}_{-1}}$
 be homomorphisms of $\C^{n+2}$ corresponding to decompositions $$\C^{n+2}=L_{1}\oplus L_{0}\oplus L_{-1}=\hat{L}_{1}\oplus \hat{L}_{0}\oplus
 \hat{L}_{-1}$$ with $L_{\pm 1}$ and $\hat{L}_{\pm 1}$ null lines and
 $L_{0}=(L_{1}\oplus L_{-1})^{\perp}$, $\hat{L}_{0}=(\hat{L}_{1}\oplus
 \hat{L}_{-1})^{\perp}$. Suppose $\mathrm{Ad}\,\gamma$ and $\mathrm{Ad}\,\hat{\gamma}$ have simple poles. Suppose as well that $\xi$ is a map into
 $O(\C^{n+2})$ holomorphic near $0$ such that $L_{1}=\xi
 (0)\hat{L}_{1}$. Then $\gamma\xi\hat{\gamma}^{-1}$ is holomorphic and invertible at $0$.
\end{lemma}

Now we proceed to the proof of Lemma \ref{rstarvshatrstar}.

\begin{proof}
For simplicity, write $p_{\mu,L}^{-1}$ and $q_{\mu,L}^{-1}$ for
$\lambda\mapsto p_{\mu,L}(\lambda)^{-1}$ and, respectively,
$\lambda\mapsto q_{\mu,L}(\lambda)^{-1}$, in the case $p_{\mu,L}$
and, respectively, $q_{\mu,L}$ are defined. As
$L^{\alpha}=q_{\beta,L^{\beta}}(\alpha)^{-1}\tilde{L}^{\alpha}_{\beta}$,
after an appropriate change of variable, we conclude, by Lemma \ref
{TrioDeHolomorfia}, that $p_{\alpha,L^{\alpha}}\,
q_{\beta,L^{\beta}}^{-1}\,
p_{\alpha,\tilde{L}^{\alpha}_{\beta}}^{-1}$ admits a holomorphic and
invertible extension to $\mathbb{P}^{1}\backslash\{\pm
\beta,-\alpha\}$. On the other hand, in view of
\eqref{twistespandq}, the holomorphicity and invertibility of
$p_{\alpha,L^{\alpha}}\, q_{\beta,L^{\beta}}^{-1}\,
p_{\alpha,\tilde{L}^{\alpha}_{\beta}}^{-1}$ at the points $\alpha$
and $-\alpha$ are equivalent. Thus $p_{\alpha,L^{\alpha}}\,
q_{\beta,L^{\beta}}^{-1}\,
p_{\alpha,\tilde{L}^{\alpha}_{\beta}}^{-1}$ admits a holomorphic and
invertible extension to $\mathbb{P}^{1}\backslash\{\pm \beta\}$, and
so does, therefore, $(p_{\alpha,L^{\alpha}}\,
q_{\beta,L^{\beta}}^{-1}\,
p_{\alpha,\tilde{L}^{\alpha}_{\beta}}^{-1})^{-1}\,
q_{\beta,L^{\beta}}^{-1}$. A similar argument shows that
$p_{\alpha,\tilde{L}^{\alpha}_{\beta}}\, (q_{\beta,L^{\beta}}\,
p_{\alpha,L^{\alpha}}^{-1}\,
q_{\beta,\hat{L}^{\beta}_{\alpha}}^{-1})$ admits a holomorphic
extension to $\mathbb{P}^{1}\backslash\{\pm\alpha\}$. But
$$p_{\alpha,\tilde{L}^{\alpha}_{\beta}}\,q_{\beta,L^{\beta}}\,
p_{\alpha,L^{\alpha}}^{-1}\,
q_{\beta,\hat{L}^{\beta}_{\alpha}}^{-1}=(p_{\alpha,L^{\alpha}}\,
q_{\beta,L^{\beta}}^{-1}\,
p_{\alpha,\tilde{L}^{\alpha}_{\beta}}^{-1})^{-1}\,
q_{\beta,L^{\beta}}^{-1}.$$ We conclude that
$p_{\alpha,\tilde{L}^{\alpha}_{\beta}}\,q_{\beta,L^{\beta}}\,
p_{\alpha,L^{\alpha}}^{-1}\,
q_{\beta,\hat{L}^{\beta}_{\alpha}}^{-1}$ extends holomorphically to
$\mathbb{P}^{1}$ and is, therefore, constant. Evaluating at
$\lambda=0$ gives
$$p_{\alpha,\tilde{L}^{\alpha}_{\beta}}\,q_{\beta,L^{\beta}}\,
p_{\alpha,L^{\alpha}}^{-1}\,
q_{\beta,\hat{L}^{\beta}_{\alpha}}^{-1}=q_{\beta,L^{\beta}}(0)\,q_{\beta,\hat{L}^{\beta}_{\alpha}}(0),$$
completing the proof.
\end{proof}

According to \eqref{twistespandq2}, $\rho\,K\,\rho=K$, showing that
$K$ preserves $V$ or, equivalently,
\begin{equation}\label{eq:KVisV}
K\,V=V.
\end{equation}
By \eqref {eq:rel}, it follows that
$$r\,^{*}(1)^{-1}\,V=\hat{r}^{*}(1)^{-1}\,V,$$ establishing a
\textit{Bianchi permutability} of type $p$ and type $q$
transformations of perturbed harmonic bundles, by means of the
commutativity of the diagram in Figure \ref{fig:im1}, below.
\begin{center}
\begin{figure}[H]
\includegraphics{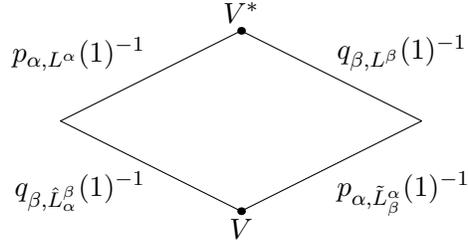}
\caption{A Bianchi permutability of type $p$ and type $q$
transformations of perturbed harmonic bundles.}\label{fig:im1}
\end{figure}
\end{center}
Equation \eqref {eq:rel} makes clear, on the other hand, that
$$\hat{r}^{*}(1)^{-1}\,\hat{r}^{*}(\infty)\,\Lambda^{1,0}=r^{*}(1)^{-1}\,r^{*}(\infty)\,\Lambda^{1,0}$$ and $$\hat{r}^{*}(1)^{-1}\,\hat{r}^{*}(0)\,\Lambda^{0,1}=r^{*}(1)^{-1}\,r^{*}(0)\,\Lambda^{0,1}.$$
We conclude that, despite not coinciding, $r^{*}$ and $\hat{r}^{*}$
produce the same transforms of perturbed harmonic bundles and constrained 
Willmore surfaces. As a final remark, note that, yet again by
equation \eqref{eq:rel},
$$\hat{q}^{*}:=\mathrm{Ad}_{\hat{r}^{*}(1)^{-1}}(\mathrm{Ad}_{\hat{r}^{*}(0)}q^{1,0}+
\mathrm{Ad}_{\hat{r}^{*}(\infty)}q^{0,1})=q^{*}.$$

\subsubsection{Real B\"{a}cklund transformation}
As we verify next, B\"{a}cklund transformation preserves reality
conditions, for special choices of parameters.\newline

Suppose $V$ is a real $q$-constrained harmonic bundle. Obviously,
the reality of $V$ establishes that of $\rho$ and, therefore,
\begin{equation}\label{eq:conj}
\overline{p_{\mu,L}(\lambda)}=p_{\overline{\mu},\overline{L}}\,(\overline{\lambda}),\,\,\,\,\,\,\,\,\,\,\overline{q_{\mu,L}(\lambda)}=q_{\overline{\mu},\overline{L}}\,(\overline{\lambda}),
\end{equation}
for all $\mu,L$ and $\lambda\in\C\backslash\{\pm\mu\}$.

\begin{lemma}
Suppose $\alpha\in\C\backslash(S^{1}\cup\{0\})$. Then we can choose
$\beta=\overline{\alpha}\,^{-1}$ and
$L^{\beta}=\overline{L^{\alpha}}$ and both $r^{*}$ and $\hat{r}^{*}$
are defined.
\end{lemma}

\begin{proof}
The reality of $\rho$ makes it clear that the non-orthogonality of
$L^{\alpha}$ and $\rho L^{\alpha}$ establishes that of
$\overline{L^{\alpha}}$ and $\rho\overline{L^{\alpha}}$, as well as,
together with \eqref{eq:conj} and \eqref{eq:conjap}, that, if $$\rho
p_{\alpha,L^{\alpha}}(\overline{\alpha}\,^{-1})\overline{L^{\alpha}}\cap
p_{\alpha,L^{\alpha}}(\overline{\alpha}\,^{-1})\overline{L^{\alpha}}=\{0\},$$
then $$\rho
q_{\overline{\alpha}\,^{-1},\overline{L^{\alpha}}}(\alpha)L^{\alpha}\cap
q_{\overline{\alpha}\,^{-1},\overline{L^{\alpha}}}(\alpha)L^{\alpha}=\{0\}.$$
On the other hand, the reality of $V$ establishes that of
$\mathcal{D}_{V}$ and $\mathcal{N}_{V}$, so that, by the reality of
$q$,
$$d^{\overline{\alpha}\,^{-1},q}_{V}=\overline{d^{\alpha,q}_{V}}.$$
Hence the $d^{\alpha,q}_{V}$-parallelness of $L^{\alpha}$
establishes the $d^{\overline{\alpha}\,^{-1},q}_{V}$-parallelness of
$\overline{L^{\alpha}}$. Obviously, if $\alpha$ is non-unit, then
$\overline{\alpha}\,^{-1}\neq\pm\alpha$. We are left to verify that
we can choose $L^{\alpha}$ a $d^{\alpha,q}_{V}$-parallel null line
subbundle of $\underline{\C}^{n+2}$ such that, locally, $\rho
L^{\alpha}\cap L^{\alpha}=\{0\}$ and $$\rho
p_{\alpha,L^{\alpha}}(\overline{\alpha}\,^{-1})\overline{L^{\alpha}}\cap
p_{\alpha,L^{\alpha}}(\overline{\alpha}\,^{-1})\overline{L^{\alpha}}=\{0\}.$$
For this, let $v$ and $w$ be sections of $V$ and $V^{\perp}$,
respectively, with $(v,v)$ never-zero, $(v,\overline{v})=0$ and
$(w,w)=-(v,v)$. Define a null section of $\underline{\C}^{n+2}$ by
$l^{\alpha}:=v+w$ and then $L^{\alpha}\subset\underline{\C}^{n+2}$
by $d^{\alpha,q}_{V}$-parallel transport of $l^{\alpha}_{p}$, for
some point $p\in \Sigma$.
\end{proof}

Let us focus then on the particular case of B\"{a}cklund
transformation of parameters $\alpha, \beta, L^{\alpha}, L^{\beta}$
with
$$\alpha\in\C\backslash(S^{1}\cup\{0\}),\,\,\,\beta=\overline{\alpha}\,^{-1},\,\,\,L^{\beta}=\overline{L^{\alpha}},$$
which we refer to as \textit{B\"{a}cklund transformation of
parameters $\alpha, L^{\alpha}$}. For this particular choice of
parameters, we write $\tilde{L}^{\alpha}$ and $\hat{L}_{\alpha}$ for
$\tilde{L}^{\alpha}_{\beta}$ and $\hat{L}^{\beta}_{\alpha}$,
respectively. Note that, by \eqref{eq:conjap} and \eqref{eq:conj},
$\overline{\hat{L}_{\alpha}}=\tilde{L}^{\alpha}$. On the other hand,
$$\overline{r^{*}(1)^{-1}}=\overline{p_{\alpha,L^{\alpha}}(1)^{-1}}\,\overline{q_{\beta,\hat{L}_{\alpha}}(1)^{-1}}=q_{\beta,L^{\beta}}(1)^{-1}p_{\alpha,\overline{\hat{L}_{\alpha}}}(1)^{-1},$$
whilst, by \eqref{eq:rel},
$$r^{*}(1)^{-1}=(K^{-1}\hat{r}^{*}(1))^{-1}=q_{\beta,L^{\beta}}(1)^{-1}p_{\alpha,\tilde{L}^{\alpha}}(1)^{-1}K.$$
Hence
\begin{equation}\label{eq:conjugado de rstar(1)inv}
\overline{r^{*}(1)^{-1}}=r^{*}(1)^{-1}K^{-1}.
\end{equation}
By \eqref{eq:KVisV}, it follows that $\overline{V^{*}}=V^{*}$. Next
we establish the reality of $q^{*}$. Yet again by \eqref{eq:conj},
$$\overline{r^{*}(0)}=q_{\alpha^{-1},\overline{\hat{L}_{\alpha}}}(0)=p_{\alpha^{-1},\overline{\hat{L}_{\alpha}}}(\infty)$$
and, on the other hand, by \eqref{eq:rel},
$$r^{*}(\infty)=K^{-1}p_{\alpha,\tilde{L}_{\alpha}}(\infty),$$ so that
\begin{equation}\label{eq:fgha}
\overline{r^{*}(0)}=Kr^{*}(\infty).
\end{equation}
Together with \eqref{eq:conjugado de rstar(1)inv}, this makes clear
that $\overline{(q^{*})^{1,0}}=(q^{*})^{0,1}$, the reality of $q$
establishes that of $q^{*}$. We conclude that:
\begin{thm}\cite{BQ}
If $V$ is a real $q$-perturbed harmonic bundle, then the
B\"{a}cklund transform $V^{*}$ of $V$, of parameters
$\alpha,L^{\alpha}$, is a real $q^{*}$-perturbed harmonic bundle.
\end{thm}

A real transformation on the level of constrained Willmore surface
follows. Equation \eqref{eq:rel} plays, yet again, a crucial role,
by showing that
$$r^{*}(1)=K^{-1}q_{\alpha^{-1},\tilde{L}^{\alpha}}(1)p_{\overline{\alpha},L^{\beta}}(1)=K^{-1}\overline{q_{\beta,\hat{L}_{\alpha}}(1)p_{\alpha,L^{\alpha}}(1)}=K^{-1}\overline{r^{*}(1)},$$
and, therefore,
\begin{equation}\label{eq:outrarK}
\overline{r^{*}(1)^{-1}}=r^{*}(1)^{-1}\overline{K}.
\end{equation}

Suppose $(\Lambda^{1,0},\Lambda^{0,1})$ is a real surface, so that,
in particular, $\overline{\Lambda^{1,0}}=\Lambda^{0,1}$. By
\eqref{eq:fgha} and \eqref{eq:outrarK}, it follows that
$\overline{(\Lambda^{*})^{1,0}}=(\Lambda^{*})^{0,1}$, establishing
the reality of the bundle
$$\Lambda^{*}:=(\Lambda^{*})^{1,0}\cap (\Lambda^{*})^{0,1}.$$

We conclude that:

\begin{thm}\cite{BQ}\label{backwillm}
If $\Lambda$ is a real $q$-constrained Willmore surface, then the
B\"{a}cklund transform $\Lambda^{*}$ of $\Lambda$, of parameters
$\alpha,L^{\alpha}$, is a real $q^{*}$-constrained Willmore surface.
\end{thm}

\subsection{Spectral deformation versus B\"{a}cklund transformation}

B\"{a}cklund transformation and spectral deformation permute, as
follows:

\begin{thm}\cite{BQ}
Let $\alpha$,$\beta$,$L^{\alpha}$,$L^{\beta}$ be B\"{a}cklund
transformation parameters to $V$, $\lambda\in\C\backslash\{0,
\pm\alpha,\pm\beta\}$ and
$$\phi^{\lambda}:(\underline{\C}^{n+2},d^{\lambda,q}_{V})\rightarrow
(\underline{\C}^{n+2},d)$$ be an isometry of bundles preserving
connections. The B\"{a}cklund transform of parameters
$\frac{\alpha}{\lambda},\frac{\beta}{\lambda}$, $\phi^{\lambda}
L^{\alpha}$, $\phi^{\lambda} L^{\beta}$ of the spectral deformation
$\phi^{\lambda}V$ of $V$, of parameter $\lambda$, corresponding to
the multiplier $q$, coincides with the spectral deformation of
parameter $\lambda$, corresponding to the multiplier $q^{*}$, of the
B\"{a}cklund transform of parameters
$\alpha,\beta,L^{\alpha},L^{\beta}$ of $V$. Furthermore, if
$$\phi^{\lambda}_{*}:(\underline{\C}^{n+2},d^{\lambda,q^{*}}_{V^{*}})\rightarrow
(\underline{\C}^{n+2},d)$$ is an isometry preserving connections,
then the diagram in Figure \ref{fig:im2} commutes.
\begin{center}
\begin{figure}[H]
\includegraphics{im.2}
\caption{A Bianchi permutability of spectral deformation and
B\"{a}cklund transformation of constrained Willmore
surfaces.}\label{fig:im2}
\end{figure}
\end{center}
\end{thm}

\begin{proof}
It is trivial, noting that
$\phi^{\lambda}r^{*}(\lambda)^{-1}r^{*}(1):(\underline{\C}^{n+2},d^{\lambda,q^{*}}_{V^{*}})\rightarrow
(\underline{\C}^{n+2},d)$ is an isometry of bundles preserving
connections.
\end{proof}

For $\lambda\in\{\pm\alpha,\pm\beta\}$, it is not clear how the
spectral deformation of parameter $\lambda$ relates to the
B\"{a}cklund transformation of parameters
$\alpha,\beta,L^{\alpha},L^{\beta}$.

\subsection{Isothermic surfaces under constrained Willmore transformation}

The isothermic surface condition is known to be preserved under
constrained Willmore spectral deformation:

\begin{prop}\cite{SD}
Constrained Willmore spectral deformation preserves the isothermic
surface condition.
\end{prop}

Next we derive it in our setting.
\begin{proof}
Suppose $(\Lambda,\eta)$ is an isothermic $q$-constrained Willmore
surface, for some $\eta,q\in\Omega^{1}(\Lambda\wedge\Lambda^{(1)})$.
Fix $\lambda\in S^{1}$ and
$\phi^{\lambda}_{q}:(\underline{\R}^{n+1,1},d^{\lambda}_{q})\rightarrow
(\underline{\R}^{n+1,1},d)$ an isometry preserving connections. Set
$$\eta_{\lambda}:=\lambda^{-1}\eta^{1,0}+\lambda\,\eta^{0,1}.$$To
prove the theorem, we show that
$(\phi^{\lambda}_{q}\Lambda,\mathrm{Ad}_{\phi^{\lambda}_{q}}\eta_{\lambda})$
is isothermic. Obviously, the reality of $\eta$ establishes that of
$\mathrm{Ad}_{\phi^{\lambda}_{q}}\eta_{\lambda}$. Recall
\eqref{eq:dlambdacoincideswithdonLambda} to conclude that
$$(\phi^{\lambda}_{q}\Lambda)^{(1)}=\phi^{\lambda}_{q}\Lambda^{(1)}$$
and, therefore, that
$\mathrm{Ad}_{\phi^{\lambda}_{q}}\eta_{\lambda}$ takes values in
$\phi^{\lambda}_{q}\Lambda\wedge (\phi^{\lambda}_{q}\Lambda)^{(1)}$.
According to \eqref{eq:LieLambdawedgeLambda1}, we have $$[q^{1,0}\wedge\eta^{0,1}]=0=[q^{0,1}\wedge \eta^{1,0}]$$ and,
therefore,
\begin{eqnarray*}
d^{d^{\lambda}_{q}}\eta_{\lambda}&=&d^{\mathcal{D}}\eta_{\lambda}+[(\lambda^{-1}\mathcal{N}^{1,0}+\lambda\mathcal{N}^{0,1}+(\lambda^{-2}-1)q^{1,0}+(\lambda^{2}-1)q^{0,1})\wedge\eta_{\lambda}]
\\&=&d^{\mathcal{D}}\eta_{\lambda}+[\mathcal{N}\wedge\eta].
\end{eqnarray*}
According to the decomposition \eqref{eq:SSperpparallelornot}, we
conclude that
$$d(\mathrm{Ad}_{\phi^{\lambda}_{q}}\eta_{\lambda})=\phi^{\lambda}_{q}\circ
d^{d^{\lambda}_{q}}\eta_{\lambda}\circ (\phi^{\lambda}_{q})^{-1}$$ 
vanishes if and only if
$d^{\mathcal{D}}\eta_{\lambda}=0=[\mathcal{N}\wedge\eta]$. Remark
\ref{deta0} and Lemma \ref{withvswithoutdecomps} complete the proof.
\end{proof}

As for B\"{a}cklund transformation of isothermic constrained Willmore surfaces,
we believe it does not necessarily preserve the isothermic condition. This shall be the subject of further work. 

A very important subclass of isothermic constrained Willmore surfaces is the class of constant mean curvature surfaces in $3$-dimensional space-forms. The constancy of the mean curvature of a surface in $3$-dimensional
space-form is preserved by both constrained Willmore spectral
deformation, cf. \cite{SD}, and constrained Willmore B\"{a}cklund transformation, cf. \cite{thesis}, for special choices of
parameters, with preservation of both the space-form and the mean
curvature in the latter case. However, constant mean curvature surfaces are not conformally-invariant objects, requiring that we carry a distinguished space-form. This shall be the subject of a forthcoming paper. See \cite[Section~8.2]{thesis} and \cite{QS} for further details.

\end{document}